\documentclass[12pt,reqno]{amsart}
\usepackage{hyperref}
\usepackage{amsthm,amsfonts,latexsym,amssymb} 
\newtheorem{lemma}{\bf{Lemma} }[section]
\newtheorem{proposition}{\bf{Proposition}}[section]
\newtheorem{theorem}{\bf{Theorem}}[section]
\newtheorem{remark}{\sc{Remark} }[section]
\newtheorem{definition}{\sc{Definition} }[section]
\newtheorem{corollary}{\bf{Corollary} }[section]

\def\build#1_#2{\mathrel{\mathop{\kern 0pt#1}\limits_{#2}}}
  
\usepackage{siunitx}

\begin{document}

 \title[Thermoelectric problem with power-type boundary effects]{Weak solutions for  a thermoelectric problem with power-type boundary effects}

\author{Luisa Consiglieri}
\thanks{Dedicated to my colleague and friend Eduardo Corregedor Borges Pires}
\address{Luisa Consiglieri, Independent Researcher Professor, European Union}
\urladdr{\href{http://sites.google.com/site/luisaconsiglieri}{http://sites.google.com/site/luisaconsiglieri}}

\begin{abstract} 
This paper deals with thermoelectric problems including the Peltier and Seebeck
effects. The coupled elliptic and doubly quasilinear parabolic equations for the electric and heat currents are stated,
respectively,
accomplished with power-type boundary conditions that describe the thermal radiative effects.
To verify the existence of weak solutions to this coupled problem (Theorem 1),
analytical investigations for abstract  multi-quasilinear elliptic-parabolic systems with nonsmooth data are presented
(Theorem 2 and 3). They are essentially approximated solutions 
 based on the Rothe method. It consists on introducing  time discretized problems,
 establishing their existence, and
then passing to the limit as the time step goes to zero.
The proof of the existence of time discretized solutions relies
 on fixed point and compactness arguments. In this study, we establish quantitative estimates to clarify
 the smallness conditions.
\end{abstract}


\maketitle

\section{Introduction}

The study of the heat equation with constant coefficients 
is a simplification from both mathematical and engineering points of view.
From the real world point of view, constant coefficients are not appropriate 
because the density and the thermal conductivity both depend on the temperature
 itself, and often also on the spatial variable.
 The concern of discontinuous leading coefficient is being
a long matter of study in the mathematical literature, as long as the  works  \cite{morr,stamp60}.
The complete concern is achieved by the doubly quasilinear parabolic equation \cite{alt,beni,caff}.
It is well known that the determination of estimates is the crucial key in the theory of
partial differential equations (PDE), which
involve the so-called universal bounds. With their abstract form, these bounds are only
qualitative and they do not have any practical use on the real world applications. In their
majority, if the proof of estimates should be remade step by step, 
the expression of the qualitative bounds would be truly cumbersome, or even impossible if the contradiction
argument is applied.
 Also regularity estimates have  being a subject of study in the last decades \cite{disser,haller-knees,Ivanov-1982},
  but these ones only occur by admitting data smoothness. 
With this in mind,
our main objective is to find quantitative estimates, \textit{i.e.} their involved constants have an explicit expression,
that are useful on the real applications.
In particular, the quantitative estimates clarify the smallness conditions on the data when 
a fixed point argument is used.

 For  the two-dimensional space situation,
a first attempt on the finding smallness conditions that assure the
 existence and regularity results for some thermoelectric problems
is presented in \cite{anona,lap2017}, where some domain dependent constants
were kept abstract. Indeed, the central existence result  of a weak solution for a class of elliptic systems
on divergence form,
 which  is only 2D valid, is provided under some higher regularity ($W^{1,p}$  regularity, with $p > 2$). The
Gehring-type higher integrability  technique
 makes the smallness conditions quite bizarre. Here, we establish more elegant smallness conditions
 and they are extended to the $n$-dimensional space situation, by finding  weak solutions.
The present model also extends  the thermal effects,  of the previous works \cite{anona,lap2017}, to the unsteady state.

Existence  of solutions for  parabolic-elliptic systems with
nonlinear no-flux boundary conditions  is not a new idea if taking constant coefficients into account \cite{biler}.
Application  of elliptic PDE system in divergence form with 
Dirichlet boundary conditions in doubly-connected domain of the plane 
 are given in \cite{cim} to the problem of electrical heating of a conductor 
 whose thermal and electrical conductivities depend on the temperature
  and to the flow of a viscous fluid in a porous medium, taking into account the Soret and Dufour effects.
In \cite{colli}, the authors deal with a traditional RLC circuit in which a
thermistor has been inserted, representing
the microwave heating process with temperature-induced modulations
on the electric field. In particular, the existence of a solution to a coupled system of three 
differential equations  (an ODE, an elliptic equation and a nonlinear parabolic PDE) 
and appropriate initial and boundary conditions is proved.
A one-dimensional thermal analysis for the performance
of thermoelectric cooler is  conducted in \cite{huang} under
the influence of the Thomson effect, the Joule heating, the Fourier heat conduction, and the radiation and
convection heat transfer.
 Simulation studies have been performed to investigate the thermal balance affected by anode shorting
 in an aluminum reduction cell \cite{cheu}. 

The method of discretization in time, whose basic idea (coming from the implicit Euler
formula) was investigated by Rothe, is a very well-known effective technique for both 
theoretical and numerical analysis, \cite{igbida,kacur,karel} and \cite{mikula,yuan}, respectively (see
also the pioneering work  \cite{alt} of Alt and Luckhaus).
 
This paper is organized as follows. 
The thermoelectric (TE) model is introduced in Section \ref{temodel}. After discussing the physical model,
the main result with respect to this model
is formulated with a  detailed description of the relevant constants.
In Section \ref{sabst}, one abstract model related to the problem under consideration is introduced 
to simplify the proofs of the existence results of time-discretized solutions (Section \ref{sabst}),
 and their corresponding steady-state solutions (Section \ref{sfpt}).
Indeed, 
 the analysis of the problem is structured via 
 two different approaches  to exemplify  alternative assumptions on the data smallness,
namely the existence results of time-discretized solutions (Subsections \ref{sabst1} and \ref{sabst2}),
 and their corresponding steady-state solutions (Subsections \ref{fpt1} and \ref{fpt2}).

\section{The thermoelectric model}
\label{temodel}

 Let $[0, T] \subset {\mathbb R}$ be the time interval with $ T >0$  being an arbitrary (but preassigned) time. 
Let $\Omega$ be a bounded domain  (that is, connected open set) in $\mathbb{R}^n$ ($n\geq 2$).
Its boundary is constituted by two disjoint open $(n-1)$-dimensional sets
 $\partial\Omega =\overline{\Gamma_\mathrm{N}}\cup \overline{\Gamma}$.
We consider $\Gamma_{\rm N}$ over which the Neumann boundary condition is taken into account,
and $\Gamma$ over which the radiative effects may occur.
Each one, $\Gamma_{\rm N}$ and $\Gamma$, may be alternatively of zero $(n-1)$-Lebesgue
measure. Set $Q_T=\Omega\times ]0,T[$ and $\Sigma_T=\Gamma\times ]0,T[$.

The electrical current density $\bf j$ and  the energy flux density ${\bf J}={\bf q}+\phi{\bf j}$, with 
$\bf q$ being the heat flux vector, are given by the constitutive relations
(see \cite{anona} and the references therein)
\begin{eqnarray}\label{pheno1}
{\bf q}&= &- k  (\cdot,\theta) \nabla\theta-\Pi(\cdot,\theta) \sigma(\cdot,\theta) \nabla\phi;\\
{\bf j}&=& -\alpha_{\rm S}(\cdot,\theta) \sigma(\cdot,\theta) \nabla\theta-
\sigma(\cdot,\theta) \nabla\phi .\label{pheno2}
\end{eqnarray} 
Here, $\theta$ denotes the absolute temperature, $\phi$ is the electric potential, 
$\alpha_{\rm S}$ represents the Seebeck coefficient,
and the Peltier coefficient $\Pi(\theta)=\theta\alpha_{\rm s}(\theta)$ is
due to  the first Kelvin relation.
The electrical conductivity $\sigma$, and
the thermal conductivity $k=k_{\rm T}+\Pi\alpha_{\rm s}\sigma$, with $k_T$ denotes the
purely conductive contribution,
are, respectively, the known positive coefficients of Ohm and Fourier laws.

The Seebeck coefficient $\alpha_{\rm S}$ has a constant sign corresponding to the Hall effect.
With positive sign ($\alpha_{\rm S}>0$), there are as examples:
the alkali metals Li, Rb and Cs \cite[p. 17]{bana}, and the noble metals Ag and Au  
\cite[p. 49, 192]{bana} or \cite[p. 71]{mcd}.
With negative sign ($\alpha_{\rm S}<0$), there are as examples:
the alkali metals Na and K \cite[p. 97]{mcd}, the transition metals Fe and Ni  \cite[p. 215]{bana},
 and the semiconductor Pb \cite[p. 48]{bana}.
We refer to \cite[p. 3]{lap2017}, and the references therein, for more examples and 
their increase and decrease behaviors.

Although heat generation starts instantaneously when the current begins to flow,
it takes time before the heat transfer process is initiated to allow the transient conditions to disappear.
Thus, the electrical current density $\bf j$ and  the energy flux density ${\bf J}$
 satisfy
\begin{eqnarray}\label{jj1}
&&\left\{\begin{array}{ll}\nabla\cdot{\bf j}=0 &\mbox{ in }\Omega\\
-{\bf j}\cdot{\bf n}=g &\mbox{ on }\Gamma_{\rm N}\\
{\bf j}\cdot{\bf n}=0 &\mbox{ on }\Gamma
\end{array}\right.\\  \label{jj2}
&&\left\{\begin{array}{ll}
\rho(\cdot,\theta) c_\mathrm{v}(\cdot, \theta)\partial_t \theta -
\nabla\cdot{\bf J}=0 &\mbox{ in } Q_T\\
{\bf J}\cdot{\bf n}=0 &\mbox{ on }\Gamma_{\rm N}\times ]0,T[\\
-{\bf J}\cdot{\bf n}=\gamma(\cdot,\theta) |\theta|^{\ell-2}\theta-h &\mbox{ on }\Sigma_T,
\end{array}\right.
\end{eqnarray}
for $\ell\geq 2$.
Here,  $\rho$ denotes the density, $c_\mathrm{v}$ denotes  the heat capacity (at constant volume),
 $\bf n$ is the unit outward normal to the boundary $\partial\Omega$,
and $g$ denotes the surface current source,

The boundary operators, $\gamma$ and   $h$,  are temperature dependent functions that
express, respectively, the radiative convection  
depending on the wavelength, and the  external heat sources.
For $\ell=5$,  the Stefan-Boltzmann radiation  law says that
$\gamma(T)=\sigma_{\rm SB}\epsilon (T)$ and 
$h(T)=\sigma_{\rm SB} \alpha(T) \theta_\mathrm{e}^{\ell-1}$, where
$\sigma_{\rm SB}=\, 5.67\times 10^{-8}$W m$^{-2} $ K$^{-4}$
 is the Stefan-Boltzmann constant for blackbodies, and
 $ \theta_\mathrm{e}$ denotes an external temperature.
The parameters, the emissivity $\epsilon$ and the absorptivity $\alpha$, both depend on the space variable
and the temperature function $\theta$.
If $\ell=2$, the boundary condition corresponds to the Newton law
of cooling with heat transfer coefficient $\gamma=h/  \theta_\mathrm{e}^{\ell-1}$.

In the framework of Sobolev and Lebesgue functional spaces, we 
use the following spaces of test functions:
\begin{eqnarray*}
V &=& \left\lbrace\begin{array}{l}
V(\Omega) = \left\{ v\in H^{1}(\Omega):\ \int_{\Omega} v \mathrm{dx}=0 \right\}\\
V(\partial\Omega) = \left\{ v\in H^{1}(\Omega):\ \int_{\partial\Omega} v \mathrm{ds}=0 \right\}
\end{array}\right. \\
V_{\ell}(\Omega) &=&\left\{v\in H^{1}(\Omega) :\ v|_{\Gamma }\in L^{\ell}(\Gamma) \right\};\\
V_{\ell}(Q_T) &=&\left\{v\in L^2(0,T; H^{1}(\Omega)) :\ v|_{\Sigma_T }\in L^{\ell}(\Sigma_T) \right\},
\end{eqnarray*}
with their usual norms,  $\ell>1$. Notice that $
V_{\ell}(\Omega)\equiv H^{1}(\Omega) $ if $\ell\leq 2_*$, where $ 2_*$ is the critical trace exponent, \textit{i.e.}
$2_*=2(n-1)/(n-2)$ if $n>2$
and $2_*>1$ is arbitrary if $n=2$.

In the presence of the previous considerations, the temperature-potential pair does not be expectable to be regular
nor even bounded.
The thermoelectric problem is formulated as follows.

\noindent (TE)
 Find the temperature-potential pair
$(\theta,\phi)$ such that  if it verifies the variational problem:
\begin{eqnarray}
\int_0^T\langle \rho(\cdot, \theta) c_\mathrm{v}(\cdot, \theta)\partial_t\theta, v\rangle\mathrm{dt} +
\int_{Q_T}  k (\cdot,\theta)\nabla\theta \cdot \nabla v\mathrm{dx}\mathrm{dt} +
\nonumber\\ +\int_{Q_T} \sigma(\cdot,\theta)\left(
T_\mathcal{M}(\phi)\alpha_{\rm S}(\cdot,\theta)\nabla\theta+(\Pi(\cdot,\theta)+T_\mathcal{M}(\phi))
\nabla\phi\right) \cdot\nabla v\mathrm{dx} \mathrm{dt}+\nonumber \\
+\int_{\Sigma_T} \gamma (\cdot,\theta)|\theta|^{\ell-2} \theta v\mathrm{ds} \mathrm{dt}
= \int_{\Sigma_T} h(\cdot,\theta) v\mathrm{ds} \mathrm{dt}; \label{pbtt} \\
\int_\Omega \sigma(\cdot,\theta)\nabla\phi \cdot \nabla w\mathrm{dx}
+\int_\Omega \sigma(\cdot,\theta) \alpha_{\rm S}(\cdot,\theta)\nabla\theta\cdot \nabla w\mathrm{dx}=
\int_{\Gamma_{\rm N}}g w\mathrm{ds}, \ \mbox{ a.e. in } ]0,T[ , \label{pbph}
\end{eqnarray}
for every $v\in V_{\ell}(Q_T)$ and $w\in V$, where 
 $\langle\cdot,\cdot\rangle$ accounts for the duality product, and 
  $T_\mathcal{M}$ is the $\mathcal{M}$-truncation function defined by $T_\mathcal{M}(z)=
\max(-\mathcal{M},\min(\mathcal{M},z))$.

We assume the following.
  
(H1) The density and the heat capacity $\rho, c_\mathrm{v}: \Omega\times\mathbb{R}\rightarrow\mathbb{R}$ are 
  Carath\'eodory functions, 
{\em i.e.} measurable with respect to $x\in\Omega$ and  continuous with respect to $e\in\mathbb R$.
  Furthermore, they verify
  \begin{equation}
   \exists b^\#,b_\#>0: \quad   b_\#\leq \rho(x,e)c_\mathrm{v}(x,e)\leq b^\# ,
\quad\mbox{for a.e. } x\in \Omega,\quad\forall e \in \mathbb{R} . \label{rmm}
  \end{equation}
  
(H2) The thermal and electrical conductivities  $k,\sigma:\Omega\times\mathbb{R}\rightarrow\mathbb{R}$ are
  Carath\'eodory functions. Furthermore, they verify
  \begin{eqnarray}
  \exists k^\#,k_\#>0:&&
k_\#\leq k(x,e)\leq k^\#;\label{kmm}\\
  \exists \sigma^\#, \sigma_\#>0:&&
\sigma_\#\leq \sigma(x,e)\leq\sigma^\#
\quad\mbox{for a.e. } x\in \Omega,\quad\forall e \in \mathbb{R}. \label{smm}
  \end{eqnarray}
  
(H3) The  Seebeck and Peltier coefficients $\alpha_{\rm S},\Pi:\Omega\times\mathbb{R}\rightarrow\mathbb{R}$ are
  Carath\'eodory functions such that
   \begin{eqnarray}
\exists\alpha^\#>0: &&
|\alpha_{\rm S}(x,e)|\leq \alpha^\#; \label{amax}\\
\exists\Pi^\#>0: &&
|\Pi(x,e)|\leq \Pi^\#,\quad\mbox{for a.e. } x\in \Omega,
\quad\forall e\in \mathbb R.\label{pmax}
  \end{eqnarray}

(H4)   The boundary function  $h$ belongs to   $L^{\ell '}(\Sigma_T)$.

(H5)   The boundary function $g$ belongs to $ L^{2}(\Gamma_{\rm N})$.

  (H6)  The boundary operator $\gamma$ is a Carath\'eodory  function from $\Sigma_T\times\mathbb{R}$
 into $\mathbb{R}$ such that
 \begin{equation}
\exists \gamma_\#, \gamma^\#>0:\quad \gamma_\#\leq \gamma(x,t,e)\leq 
\gamma^\#; \quad\mbox{for a.e. } (x,t)\in \Sigma_T,\quad\forall e\in \mathbb R. \label{gamm}
  \end{equation}
   Moreover,  $\gamma$ is strongly monotone:
\[
 \left(\gamma( u )| u |^{\ell-2} u- \gamma( v )| v |^{\ell-2} v \right) ( u- v)\geq 
\gamma _\# |u - v|^\ell .
\]

Let us state our main existence theorem.
\begin{theorem}\label{tmain}
Let (H1)-(H6) be fulfilled. The thermoelectric problem (TE) admits a solution
 $(\theta,\phi)\in V_\ell(Q_T)\times L^2(0,T; V)$, for $\mathcal{M}$ being such that
 \begin{equation}\label{akM}
   \mathcal{M}\alpha^\#\sigma^\#< k_\#,
\end{equation}
 and
  one of the following hypothesis is assured:
\begin{enumerate}
\item there holds
\begin{equation}
4( k_\# - \mathcal{M}\alpha^\#\sigma^\# ) \sigma_\# > (\sigma^\#)^2 ( \Pi^\#+\mathcal{M}  +\alpha^\#)^2 ;
\label{sss1} 
\end{equation}
\item there holds
\begin{equation}
4( k_\# - \mathcal{M}\alpha^\#\sigma^\# ) >\sigma^\# ( \Pi^\#+\mathcal{M}  +\alpha^\#)^2 ;
\label{sss2} 
\end{equation}
\item  there holds
\begin{equation}\label{sss3}
 k_\#> \sigma^\# \alpha^\# ( 2 \Pi^\#+ 3 \mathcal{M}) .
\end{equation}
\end{enumerate}
\end{theorem}

\section{Existence of approximated solutions}
\label{sabst}

The thermoelectric problem provides the abstract initial boundary value problem
\begin{eqnarray}\label{strong}
b(\theta)\partial_t \theta -\nabla\cdot \left( a  ( \theta,\phi )\nabla \theta \right)
=  \nabla\cdot(  \sigma( \theta ) F   (\theta,\phi)\nabla\phi) && \\
-\nabla\cdot(\sigma( \theta)\nabla\phi )= \nabla\cdot
\left(  \sigma( \theta ) \alpha_\mathrm{S} (\theta) \nabla \theta\right) && \mbox{ in }Q_T ; \\
\label{equal}
 \left( a  ( \theta,\phi  )\nabla \theta+  \sigma( \theta )  F  ( \theta,\phi )\nabla\phi \right)
\cdot \mathbf{n} =\left(h - \gamma(\theta) |\theta|^{\ell-2} \theta \right)\chi_{\Gamma}  && \\
\left( \sigma( \theta )\nabla \phi +   \sigma( \theta ) \alpha_\mathrm{S}  (\theta)\nabla \theta\right)\cdot{\bf n} =
g \chi_{\Gamma_\mathrm{N}}
&&\mbox{ on } \partial\Omega\times ]0,T[  . \label{wbc}
\end{eqnarray}
This abstract problem is formulated in the form that the coefficients are correlated with
the leading coefficient $\sigma$. We emphasize that this interrelation must be clear.

  Let us assume the hypothesis set.
 
  (H)
The operators $ a,F $ and $b,\sigma,  \alpha_\mathrm{S}$ are  Carath\'eodory functions
 from $\Omega\times\mathbb{R}^2$ and $\Omega\times\mathbb{R}$,  respectively, into  $\mathbb{R}$, 
   which enjoy the following properties.
There exist positive constants $F^\#, a_\#,a^\#, b_\#,b^\#$ such that
\begin{eqnarray}
 | F(x,e,d )|\leq  F^\#; &&\label{isof} \\
a_\#\leq a(x,e,d )\leq a^\#; &&\label{amm} \\
b_\#\leq b (x,e )\leq b^\# &&
\mbox{for a.e. } x\in \Omega,\quad\forall e,d \in \mathbb{R} , \label{bmm}
\end{eqnarray}
and $\sigma_\#,\sigma^\#, \alpha^\#$ verifying
 (\ref{smm}), (\ref{amax}), respectively.

\begin{definition}\label{dweak}
We say that $(\theta,\phi)$ is a weak solution to (\ref{strong})-(\ref{wbc})
if it solves the variational problem
\begin{eqnarray}
\int_0^T\langle b(\cdot, \theta)\partial_t\theta, v\rangle\mathrm{dt} +
\int_{Q_T} a (\cdot,\theta,\phi)\nabla\theta \cdot \nabla v\mathrm{dx}\mathrm{dt}
+\int_{\Sigma_T} \gamma (\cdot,\theta)|\theta|^{\ell-2} \theta v\mathrm{ds} \mathrm{dt}=
\nonumber\\ = - \int_{Q_T} 
 \sigma( \theta ) F  (\cdot,\theta,\phi)\nabla\phi  \cdot\nabla v\mathrm{dx} \mathrm{dt}
+\int_{\Sigma_T} h(\cdot,\theta) v\mathrm{ds} \mathrm{dt}; \label{wvftta} \\
\int_\Omega \sigma(\cdot,\theta)\nabla\phi \cdot \nabla w\mathrm{dx}
+ \int_\Omega  \sigma(\cdot, \theta )  \alpha_\mathrm{S}  (\cdot,\theta)\nabla\theta \cdot \nabla w\mathrm{dx} =
\int_{\Gamma_{\rm N}}g w\mathrm{ds},  \mbox{ a.e. in } ]0,T[ ,\label{wvfpha}
\end{eqnarray}
for every $v\in V_{\ell}(Q_T)$ and $w\in V$.
\end{definition}

We define an auxiliary operator. Denote by $B$ the operator from $H^1(\Omega)$ into $L^2(\Omega)$ defined by
\begin{equation}\label{defb}
B(v)=\int_0^v b(\cdot,z)\mathrm{dz},
\end{equation}
for all $v\in H^1(\Omega)$.

Different approaches in the finding of solutions according to Definition \ref{dweak} provide different
smallness conditions (\ref{afg}), (\ref{sfg}) or (\ref{asfg}). We emphasize that the difference between these
smallness conditions has its importance in the real-world applications.
\begin{theorem}\label{abst1}
Let (H) and (H4)-(H6) be fulfilled.
If there exists $\varepsilon>0$ such that one the following relation holds, that is, either
\begin{equation}
a_\# >\varepsilon\sigma^\#(F^\#+\alpha^\#)/2\quad\mbox{ and } \quad \varepsilon\sigma_\# >\sigma^\#
(F^\#+\alpha^\#)/2, 
\label{afg} 
\end{equation}
or
\begin{equation}
a_\# >\varepsilon\sqrt{\sigma^\#}(F^\#+\alpha^\#)/2 \quad\mbox{ and } \quad \varepsilon>\sqrt{\sigma^\#}
(F^\#+\alpha^\#)/2, \label{sfg}
\end{equation}
then the variational problem  (\ref{wvftta})-(\ref{wvfpha}) admits a sequence of approximate solutions
 $\lbrace(\theta_M,\phi_M)\rbrace_{M\in\mathbb{N}}$ 
  in the  sense established in Section \ref{sabst1}.
\end{theorem}
The proof of Theorem \ref{abst1} relies on the limit solution to the
 recurrent sequence of time-discretized problems
\begin{eqnarray}
\frac{1}{\tau}\int_\Omega B (\theta^m) v\mathrm{dx}+\int_\Omega
 a (\theta^m,\phi^m )\nabla \theta^{m}\cdot\nabla v\mathrm{dx}
+\int_\Gamma \gamma(\theta^m)|\theta^m|^{\ell-2} \theta^{m} v \mathrm{ds}   +\nonumber \\ +
\int_\Omega\sigma (\theta^m) F( \theta^m,\phi^m ) \nabla\phi^m \cdot\nabla v\mathrm{dx}
\label{wvfttm} = 
\frac{1}{\tau}\int_\Omega B (\theta^{m-1} )v\mathrm{dx}
+\int_{\Gamma}  h_{m} v  \mathrm{ds} ; \\
\int_\Omega\sigma (\theta^m)\nabla\phi^m\cdot\nabla w\mathrm{dx}
 + \int_\Omega \sigma (\theta^m) \alpha_\mathrm{S}(\theta^m)
\nabla\theta^{m} \cdot \nabla w \mathrm{dx}=\int_{\Gamma_\mathrm{N}}  g w \mathrm{ds} ,\label{wvfphim}
\end{eqnarray}
where $\tau$ is  the so called time step,  $B$ is defined in (\ref{defb}), $m\in \mathbb{N}$ and $h_m$
is conveniently chosen in  Section \ref{tdt} (the time discretization technique). 
We call $\phi^m$ the corresponding solution to the time independent temperature $\theta^{m}$.

\begin{theorem}\label{abst2}
Let (H) and (H4)-(H6) be fulfilled.
  If there holds 
\begin{equation}\label{asfg}
 a_\#>2\sigma^\# \alpha^\#  F^\# ,
\end{equation}
then the variational problem  (\ref{wvftta})-(\ref{wvfpha}) admits a sequence of approximate solutions
 $\lbrace(\theta_M,\phi_M)\rbrace_{M\in\mathbb{N}}$ 
  in the  sense established in Section \ref{sabst2}.
\end{theorem}
The proof of Theorem \ref{abst2} relies on the limit solution to the
 recurrent sequence of time-discretized problems
\begin{eqnarray}
\frac{1}{\tau}\int_\Omega B (\theta^m) v\mathrm{dx}+\int_\Omega
 a (\theta^m,\phi^m )\nabla \theta^{m}\cdot\nabla v\mathrm{dx}
+\int_\Gamma \gamma(\theta^m)|\theta^m|^{\ell-2} \theta^{m} v \mathrm{ds}  +\nonumber \\ +
\int_\Omega\sigma (\theta^m) F( \theta^m,\phi^m ) \nabla\phi^m \cdot\nabla v\mathrm{dx}
\label{wvfttm2} = 
\frac{1}{\tau}\int_\Omega B (\theta^{m-1} )v\mathrm{dx}
+\int_\Gamma  h_{m} v  \mathrm{ds} ; \\
\int_\Omega\sigma (\theta^{m-1})\nabla\phi^m\cdot\nabla w\mathrm{dx}
 = - \int_\Omega \sigma (\theta^{m-1})  \alpha_\mathrm{S}(\theta^{m-1})
\nabla\theta^{m-1} \cdot \nabla w \mathrm{dx} +\nonumber \\ + \int_{\Gamma_\mathrm{N}}  g w \mathrm{ds} ,\label{wvfphim2}
\end{eqnarray}
where $\tau$ is  the so called time step,  $B$ is defined in  (\ref{defb}), $m\in \mathbb{N}$ and $h_m$
is conveniently chosen in  Section \ref{tdt} (the time discretization technique).
We call $\phi^m$ the corresponding solution to the time independent temperature $\theta^{m-1}$.

\section{Steady-state solvability}
\label{sfpt}

In this section, we prove the existence of solutions to the
 recurrent sequence of time-discretized problems  (\ref{wvfttm})-(\ref{wvfphim})
 and  (\ref{wvfttm2})-(\ref{wvfphim2}) in Subsections \ref{fpt1} and \ref{fpt2}, respectively.
Since $m\in \mathbb{N}$ is fixed and $\theta^{m-1}\in V_{\ell}(\Omega)$ is given, 
 for the sake of simplicity, we set
 $f= B(\theta^{m-1})$ and $H=h_m$, and we omit the index to the unknown pair, \textit{i.e.} we
 simply write $(\theta,\phi)$.

Denoting by $K_{2}$ the continuity constant of the trace embedding $H^{1}(\Omega)\hookrightarrow L^2(\Gamma)$,
 with $2_*=2(n-1)/(n-2)$ if $n>2$, and any $2_*>2$ if $n=2$, and by
 $P_2$  the Poincar\'e constant correspondent to the space exponent $2$,
 the constant $K_2(P_2+1)$ obeys
\begin{equation}\label{k2p2}
\|v\|_{2,\Gamma}\leq K_2 \left( \|v\|_{2,\Omega}+ \| \nabla v\|_{2,\Omega}\right)
\leq K_2(P_2+1) \|\nabla v\|_{2,\Omega},
\quad \forall v\in H^1(\Omega).
\end{equation}

Let us introduce  \cite{alt,jaka} 
 \[
\Psi (s) := B(s)s-\int_0^s B(r)\mathrm{dr} = \int_0^s(B(s) - B(r) )\mathrm{dr}.
\]
 We state the main properties of the auxiliary operators $B$ and $\Psi$,
 the ones that we will use later. For completeness sake, 
we sketch the proof of the property (\ref{bpsi}).
\begin{lemma}\label{lbmm}
There holds
\begin{equation}\label{bpsi}
\int_\Omega (B(u)-B(v) ) u \mathrm{dx} \geq 
\int_\Omega \Psi(u)\mathrm{dx} - \int_\Omega \Psi(v)\mathrm{dx} .
\end{equation}
In particular, if  the assumption (\ref{bmm}) is fulfilled then there holds
\[
\int_\Omega \Psi(u)\mathrm{dx}\leq 
\int_\Omega B(u) u \mathrm{dx} \leq b^\# \| u\|_{2,\Omega}^2 .
\]
Under the assumption (\ref{bmm}) the operator $B$ verifies
\[
(B(u)-B(v),u-v)\geq b_\# \|u-v\|_{2,\Omega}^2 .
\]
\end{lemma}
\begin{proof}
Let us write the decomposition
\[
(B(u)-B(v) ) u  = B(u)u-B(v) v  - B(v)(u-v).
\]
Thanks to the mean value theorem for definite integrals, there exists $c$ between $u$ and $v$ such that
\[
 \int_{v} ^{u} B(r)\mathrm{dr}= B(c)(u-v).\]
 Since $-B$ is a decreasing function, we obtain
 \[
\int_\Omega (B(u)-B(v) ) u \mathrm{dx} \geq 
\int_\Omega (B(u)u-B(v) v )\mathrm{dx} - 
\int_\Omega \int_{v} ^{u} B(r)\mathrm{dr} \mathrm{dx},
\]
which concludes the proof  by definition of $\Psi$.
\end{proof}

Finally, we recall the following remarkable lemma \cite[Lemma 1.9]{alt}.
\begin{lemma} \label{lbmm2}
Suppose $u_m$ weakly converge to $u$ in $L^p(0,T;W^{1,p}(\Omega))$, $p>1$, with the estimates
\[
\int_\Omega \Psi(u_m(t))  \mathrm{dx}\leq C\quad \mbox{for } 0<t<T,
\]
and for $z>0$
\begin{equation}\label{cbb}
 \int_0^{T-z}
\int_\Omega (B(u_m(t+z))-B(u_m(t)) ) (u_m(t+z)-u_m(t)) \mathrm{dx} \mathrm{dt} \leq Cz ,
\end{equation}
with $C$ being  positive constants.
Then, $B(u_m) \rightarrow B(u)$ in $L^1(Q_T)$ and
$\Psi (u_m) \rightarrow \Psi(u) $ almost everywhere in  $Q_T$.
\end{lemma}

\subsection{Fixed point argument (solvability  to (\ref{wvfttm})-(\ref{wvfphim}))}
\label{fpt1}
\

Let $\ell\geq 2$, and define an operator $\mathcal{T}$ from $\mathbf{V}_{\ell} =
  V_\ell(\Omega)\times  V$ into itself  such that
$(\theta,\phi)=\mathcal{T}(\mathbf{u})$ is the unique solution of Proposition \ref{propu}.
\begin{proposition}\label{propu}
Let $\mathbf{u} =(u_1,u_2)\in \mathbf{V}_{\ell}$, and $u=u_1$.
Then, there exists a unique solution $(\theta,\phi)\in \mathbf{V}_{\ell}$ to the 
Neumann-power type elliptic problem 
\begin{eqnarray}
\frac{1}{\tau}\int_\Omega b(u)\theta v\mathrm{dx}+
\int_\Omega a( \mathbf{u} )\nabla \theta \cdot\nabla v\mathrm{dx}+
\int_\Omega \sigma (u) F( \mathbf{u}) \nabla\phi \cdot\nabla v\mathrm{dx} +\nonumber \\
+\int_\Gamma \gamma(u) |\theta|^{\ell-2} \theta v \mathrm{ds}   = 
\frac{1}{\tau}\int_\Omega fv \mathrm{dx} +\int_\Gamma  Hv \mathrm{ds} ; \label{wvfU} \\
\int_\Omega\sigma (u)\nabla\phi\cdot\nabla w\mathrm{dx}
 + \int_\Omega \sigma (u) \alpha_\mathrm{S}(u) \nabla \theta\cdot \nabla w \mathrm{dx}
=\int_{\Gamma_\mathrm{N}}  g w \mathrm{ds} ,\label{wvfphiU}
\end{eqnarray}
for all $v\in V_{\ell}(\Omega)$ and $w\in V$. In addition, the following estimate
\begin{eqnarray}
\frac{b_\#}{2\tau }\|\theta \|_{2,\Omega}^2 + (L_{1})_\#\|\nabla \theta \|_{2,\Omega}^2 + 
\frac{(L_{2})_\# }{2}\|\nabla\phi \|_{2,\Omega}^2 
+\frac{\gamma_\#}{\ell '} \| \theta  \|_{\ell,\Gamma} ^\ell \leq 
\frac{ 1 }{ 2\tau b_\# }\|f\|_{2,\Omega}^2+ \nonumber\\  +
\frac{1}{\ell '\gamma_\#^{1/(\ell-1)}} \| H\|_{\ell ',\Gamma}^{\ell '}+
\frac{(K_2)^2(P_2+1)^2}{2 (L_{2})_\#}\|g\|_{2,\Gamma_\mathrm{N} } ^2 := \mathcal{R}( \|f\|_{2,\Omega}^2 ,\| H\|_{\ell ',\Gamma}^{\ell '} )  \quad
\label{cotatphi}
\end{eqnarray}
holds true, if provided by one the following definition
 \begin{eqnarray}\label{L1c}
&&\left\{\begin{array}l
(L_{1})_\# = a_\#-  \varepsilon\sigma^\# \left( F^\# +  \alpha^\#\right) /2\\
(L_{2})_\# = \sigma_\# -  \sigma^\# \left( F^\# +  \alpha^\#\right) /(2\varepsilon)
\end{array}\right. \\  \label{L2c}
&&\left\{\begin{array}l
(L_{1})_\# = a_\#-  \varepsilon\sqrt{\sigma^\# }( F^\#+\alpha^\#) /2\\
(L_{2})_\# = \sigma_\#\left( 1- \sqrt{\sigma^\# }( F^\#+\alpha^\#) /(2\varepsilon) \right) 
\end{array}\right. .
 \end{eqnarray}
\end{proposition}
\begin{proof}
The existence of a solution to the variational system (\ref{wvfU})-(\ref{wvfphiU})
relies on the direct application of the Browder-Minty Theorem \cite{ll65}. Indeed, the form 
$\mathcal{F}: \mathbf{V}_{\ell}\rightarrow \mathbb{R}$  defined by
\[
\mathcal{F}(v,w)=\frac{1}{\tau}\int_\Omega  fv\mathrm{dx}
+\int_\Gamma  Hv \mathrm{ds} +\int_{\Gamma_\mathrm{N}}  g w \mathrm{ds} 
\]
is continuous and linear, and
the  form $\mathcal{L}: \mathbf{V}_{\ell}\times\mathbf{V}_{\ell}\rightarrow \mathbb{R}$ defined by
\begin{eqnarray*}
\mathcal{L} \left( (\theta,\phi), (v,w)\right)=\frac{1}{\tau}\int_\Omega b(u)  \theta v\mathrm{dx}+
\int_\Omega \left(\mathsf{L}( \mathbf{u} )\nabla 
\left[\begin{array}{c}
 \theta\\
 \phi
 \end{array} \right]\right)\cdot\nabla  \left[\begin{array}{c} v\\
 w
 \end{array} \right]
 \mathrm{dx},
\end{eqnarray*}
is continuous and bilinear, with $\mathsf{L}$ being the $(2\times 2)$-matrix
\[ 
\mathsf{L} ( \mathbf{u} )=\left[
\begin{array}{cc}
a( \mathbf{u} )& \sigma (u) F( \mathbf{u} ) \\
\sigma (u) \alpha_\mathrm{S} (u)
& \sigma (u)
\end{array}
\right] .
\]
Moreover, $\mathcal{L}$ is coercive:
\begin{equation}\label{Lcoer}
\sum_{i,j=1}^{ 2}\sum_{l=1}^n \left( L_{i,j} ( \mathbf{u})\xi_{j,l}\right)\xi_{l,i}\geq
(L_{1})_\#|\xi_1|^2+ (L_{2})_\#|\xi_2|^2,
\end{equation}
with $(L_{1})_\#$ and $(L_{2})_\#$   being the positive constants defined in (\ref{L1c}) or (\ref{L2c}),
  taking the assumptions (\ref{afg}) and  (\ref{sfg}) into account.
 The difference of the definitions is consequence
of the different application of the Young inequality $2AB\leq \varepsilon A^2+B^2/\varepsilon$ ($\varepsilon,A,B>0$),
see Remark \ref{young}.
Namely, with
\begin{enumerate}
\item $A=| \xi_{1}| $ and $B=| \xi_{2}|$, for (\ref{L1c}). That is,
\[ \sum_{l=1}^n \left(
\sigma (u)  F( \mathbf{u} )  \xi_{2,l}\xi_{l,1}+
\sigma (u) \alpha_\mathrm{S} (u) \xi_{1,l}\xi_{l,2} \right) \leq \sigma ^\# \left( F^\#+\alpha^\# \right)
\left( \frac{\varepsilon}{2}A^2+\frac{1}{2\varepsilon} B^2\right).
\]
\item $A=| \xi_{1}| $ and $B=\sqrt{\sigma (u)}| \xi_{2}| $,  for  (\ref{L2c}). That is, 
\[ \sum_{l=1}^n \left(
\sigma (u)  F( \mathbf{u} )  \xi_{2,l}\xi_{l,1}+
\sigma (u) \alpha_\mathrm{S} (u) \xi_{1,l}\xi_{l,2} \right)\leq \sqrt{\sigma^\# }\left( F^\#+\alpha^\# \right)
\left( \frac{\varepsilon}{2}A^2+\frac{ 1}{2\varepsilon}B^2\right).
\]
\end{enumerate}

Finally, observing that the function $e\in\mathbb{R}\mapsto \gamma(u)
|e|^{\ell-2}e $ is monotonically increasing, we conclude the existence of the required solution.
 
In order to obtain (\ref{cotatphi}), we take $v=\theta$ and $w=\phi$ as  test functions in (\ref{wvfU}) and (\ref{wvfphiU}),
respectively. Summing the obtained relations, and applying (\ref{bmm}),
(\ref{gamm}), the coercivity (\ref{Lcoer}) of $\mathsf{L}$, and the H\"older inequality, we find
\begin{eqnarray}
\frac{ b_\# }{\tau} \|\theta\|_{2,\Omega}^2+ (L_{1})_\# \|\nabla \theta\|_{2,\Omega}^2+
(L_{2})_\#\|\nabla\phi \|_{2,\Omega}^2 
+\gamma_\# \| \theta\|_{\ell,\Gamma} ^\ell \leq \nonumber \\ \leq
\frac{ 1 }{\tau} \|f\|_{2,\Omega}\|\theta\|_{2,\Omega}+
\|H\|_{\ell ',\Gamma}\|\theta\|_{\ell,\Gamma}+
\|g\|_{2,\Gamma_\mathrm{N}}\| \phi\|_{2,\Gamma_\mathrm{N}}.\label{eql1l2}
\end{eqnarray}
We successively apply (\ref{k2p2}) and the Young inequality  to obtain
\begin{eqnarray}
\| H\|_{\ell',\Gamma} \| \theta\|_{\ell,\Gamma}+
\|g\|_{2,\Gamma_\mathrm{N} }\|\phi\|_{2,\Gamma_\mathrm{N}} \leq \nonumber \\
\leq \frac{1}{\ell '\gamma_{\#}^{1/(\ell-1)} }
\| H\|_{\ell',\Gamma}^{\ell '} 
+\frac{\gamma_\#}{\ell}\| \theta\|_{\ell,\Gamma}^\ell+
 \frac{K_2^2(P_2+1)^2}{ 2(L_{2})_\#}
\|g\|_{2,\Gamma_\mathrm{N}}^2+ \frac{(L_{2})_\#}{2}\|\nabla\phi\|_{2,\Omega} ^2. \label{eqrbum}
\end{eqnarray}
Inserting (\ref{eqrbum}) into (\ref{eql1l2}), we deduce (\ref{cotatphi}).
\end{proof}

\begin{remark}\label{young}
Even $\varepsilon>0$ may be an arbitrary (but fixed) number, we may differently define $(L_1)_\#$ and $(L_2)_\#$.
Indeed, the Young inequality $2AB\leq \varepsilon A^2+B^2/\varepsilon$ ($\varepsilon,A,B>0$) may be applied to obtain
\begin{eqnarray*}
\sum_{l=1}^n \left(
\sigma (u)  F( \mathbf{u} )  \xi_{2,l}\xi_{l,1}+
\sigma (u) \alpha_\mathrm{S} (u) \xi_{1,l}\xi_{l,2} \right) \leq \sigma ^\# \left( F^\#
\left( \frac{\varepsilon_1}{2}| \xi_{1}|^2+\frac{1}{2\varepsilon_1} | \xi_{2}|^2\right)+\right. \\ \left. +
\alpha^\# 
\left( \frac{\varepsilon_2}{2}| \xi_{1}|^2+\frac{1}{2\varepsilon_2} | \xi_{2}|^2\right) \right).
\end{eqnarray*}
\end{remark}

Next, let us determine whose radius make possible that  the operator $\mathcal{T}$ maps a closed ball into itself.
\begin{proposition}\label{boundiso}
For $R=\max \{R_1,R_2\}$ with $R_1$ and $R_2$ being 
defined in (\ref{defr1}) and (\ref{defr2}), respectively,
the operator $\mathcal{T}$ verifies $\mathcal{T}(K)\subset K$, with
\begin{eqnarray*}
K=\left\lbrace (v,w)\in \mathbf{V}_\ell : \ \|\nabla w\|_{2,\Omega}+
\|\nabla v\|_{2,\Omega}+ \|v\|_{\ell,\Gamma} \leq R\right\rbrace .
\end{eqnarray*}
\end{proposition}
\begin{proof}
Let $\mathbf{u}\in \mathbf{V}_{\ell}$,  $u=u_1$ and $(\theta,\phi)$ be the unique solution of Proposition \ref{propu},
 \textit{i.e.}
 $(\theta,\phi)=\mathcal{T}(\mathbf{u})$. In order to prove that
$(\theta,\phi)\in K$ we consider two different
 cases: (1) if $\| \theta\|_{\ell,\Gamma} \leq1$;
 and (2) if $\| \theta\|_{\ell,\Gamma} >1$,
 \begin{enumerate}
 \item   if $\| \theta\|_{\ell,\Gamma} \leq 1$, then there holds
 \[\|\nabla \phi\|_{2,\Omega}+
\|\nabla \theta\|_{2,\Omega}+ \|\theta\|_{\ell,\Gamma} \leq
 \sqrt{2}\left(\|\nabla \phi\|_{2,\Omega}^2+ \|\nabla \theta\|_{2,\Omega}^2\right)^{1/2}+1,
\]
by applying the elementary inequality $(a+b)^2\leq 2(a^2+b^2)$ for every $a,b\geq 0$.
 By using (\ref{cotatphi}), we may take
\begin{equation}
R_1= 
\left(\frac{2\mathcal{R}}{\min\left\lbrace  (L_{1})_\#, (L_{2})_\#/2\right\rbrace} \right)^{1/2}+1 . \label{defr1}
\end{equation}
 \item   if $\| \theta\|_{\ell,\Gamma} > 1$, then using $\ell\geq 2$ there holds
 \[\|\nabla \phi\|_{2,\Omega}+
\|\nabla \theta\|_{2,\Omega}+ \|\theta\|_{\ell,\Gamma} \leq
 \sqrt{2}\left( 2(\|\nabla \phi\|_{2,\Omega}^2+
\|\nabla \theta\|_{2,\Omega}^2) + \|\theta\|_{\ell,\Gamma} ^\ell\right)^{1/2},
\]
by applying the elementary inequality $(a+b)^2\leq 2(a^2+b^2)$ for every $a,b\geq 0$.
 By using (\ref{cotatphi}), we may take
\begin{equation}
R_2^2= 
\left(\frac{2 }{\min\left\lbrace  (L_{1})_\#, (L_{2})_\#/2\right\rbrace} +
\frac{\ell '}{\gamma_\#}\right)\mathcal{R} . \label{defr2}
\end{equation}
 \end{enumerate}
 
Then, the proof is complete by taking $R$ such that is the maximum of $R_1$ and $R_2$
defined in (\ref{defr1}) and (\ref{defr2}), respectively. 
\end{proof}

\begin{proposition}\label{contnt}
The operator $\mathcal{T}$ is continuous.
\end{proposition}
\begin{proof}
Let $\lbrace \mathbf{u}^m\rbrace_{m\in\mathbb{N}}$ be a sequence such that
weakly converges to $\mathbf{u}=(u,u_2)$ in $ \mathbf{V}_\ell$,  and
$(\theta_m,\phi_m)=\mathcal{T}(\mathbf{u}^m)$ for each $m\in\mathbb{N}$.
Proposition \ref{propu} guarantees that $(\theta_m,\phi_m)$ solves, for each $m\in\mathbb{N}$,
 the variational system (\ref{wvfU})$_m$-(\ref{wvfphiU})$_m$, 
with  $\mathbf{u}$ replaced by  $\mathbf{u}^m$.
The uniform boundedness ensured by Proposition \ref{boundiso} guarantees the existence of a limit
$(\theta,\phi)\in \mathbf{V}_\ell$, for at least a subsequence of $(\theta_m,\phi_m)$
still denoted by $(\theta_m,\phi_m)$, such that
\[ 
\theta_m\rightharpoonup \theta \quad\mbox{in } V_\ell(\Omega)\quad \mbox{and}\quad
\phi_m\rightharpoonup \phi \quad\mbox{in } V \quad
(\mbox{as }   m\rightarrow +\infty) .
\] 

The Rellich-Kondrachov theorem guarantees the strong convergences
\begin{eqnarray*}
u^m\rightarrow u  &\mbox{ and }& u_2^m\rightarrow u_2 \quad\mbox{in } L^2(\Omega); \\
\theta_m\rightarrow \theta  &\mbox{ and }& \phi_m\rightarrow \phi \quad\mbox{in } L^2(\Omega); \\
u^m\rightarrow u  &\mbox{ and }& \theta_m\rightarrow \theta \quad\mbox{in } L^2(\Gamma).
\end{eqnarray*}
To show that  $(\theta,\phi)=\mathcal{T}(\mathbf{u})$, it remains to pass to the limit in the system
(\ref{wvfU})$_m$-(\ref{wvfphiU})$_m$ as  $m$ tends to infinity.

 Applying the Krasnoselski theorem to the Nemytskii operators $b$, $a$, $\sigma$, we have 
 \begin{eqnarray*}
 b(u^m) v\rightarrow b(u) v &\mbox{ in }& L^2(\Omega);\\
 a(\mathbf{u}^m) \nabla v\rightarrow a(\mathbf{u}) \nabla  v &\mbox{ in }&  \mathbf{L}^2(\Omega); \\
\sigma (u^m)\nabla v\rightarrow \sigma( u)\nabla v &\mbox{ in }& \mathbf{L}^2(\Omega),
 \end{eqnarray*}
 for all $v\in H^1(\Omega)$, making use of the Lebesgue dominated convergence theorem and
 the assumptions (\ref{smm}) and (\ref{amm})-(\ref{bmm}). Also the terms
 $\sigma (u^m)F(\mathbf{u}^m) \nabla v$ and $ \sigma (u^m) \alpha_\mathrm{S} (u^m) \nabla v$
  pass to the limit making recourse to the assumptions
 (\ref{isof}) and (\ref{amax}), respectively.
 
 Similarly, the boundary term $\gamma (u^m)v$ converges  to  $\gamma (u) v$ in $L^{\ell '}(\Gamma)$, for all
 $v\in L^{\ell '}(\Gamma)$, due to (\ref{gamm}). 
 Observe that $ \theta_m$ strongly converges to $\theta$ in $L^p(\Gamma)$, for all $1<p<\ell$.
 Then, the nonlinear boundary term $\gamma(u^m) |\theta^m|^{\ell-2} \theta^m $ weakly
 passes to the limit as $m$ tends to infinity to $\gamma (u)\Lambda$  in $L^{\ell '}(\Gamma)$.
Therefore, the variational system (\ref{wvfU})$_m$-(\ref{wvfphiU})$_m$
as $m$ tends to infinity to conclude that $\phi$ is the required limit solution, \textit{i.e.}
it solves the limit equality (\ref{wvfphiU}), while $\theta$ verifies
\begin{eqnarray}
\frac{1}{\tau}\int_\Omega b(u)\theta v\mathrm{dx}+
\int_\Omega a( \mathbf{u} )\nabla \theta \cdot\nabla v\mathrm{dx}+
\int_\Omega \sigma (u) F( \mathbf{u}) \nabla\phi \cdot\nabla v\mathrm{dx} +\nonumber \\
+\int_\Gamma \gamma(u) \Lambda v \mathrm{ds}   = 
\frac{1}{\tau}\int_\Omega fv \mathrm{dx} +\int_\Gamma  Hv \mathrm{ds}.\label{lambda}
\end{eqnarray}
It  remains to identify the limit $\Lambda$
 by using  the Minty trick  as follows.
The argument is slightly different from the classical one (see  \cite{ll65}).

Making recourse to the the lower bound (\ref{gamm})  of $\gamma$ and the monotone property of the function
$v\mapsto |v|^{\ell-2}v$, we have
\[
0\leq \gamma _\# 2^{2-\ell }|\theta_m -v|^\ell 
\leq \gamma(u^m) \left(|\theta_{m}|^{\ell-2} \theta_{m} -
|v|^{\ell-2}v\right) (\theta_{m} -v).
\]
 
 Thanks to the coercivity coefficients  (\ref{L1c}) or  (\ref{L2c}), the monotonicity property of the boundary term,
  and the H\"older and  Young inequalities,
let us consider
 \begin{eqnarray}
\int_{\Omega}a(\mathbf{u}^m)|\nabla (\theta_m- v)|^2\mathrm{dx} +
\int_{\Omega}\sigma(u^m)|\nabla (\phi_m-\phi)|^2\mathrm{dx}+ \nonumber \\ +
\int_{\Omega} \sigma (u^m) F (\mathbf{u}^m)\nabla (\phi_m -\phi) \cdot\nabla( \theta_m-v)\mathrm{dx}  + \nonumber \\ +
\int_{\Omega} \sigma (u^m)  \alpha_\mathrm{S} (u^m)\nabla (\theta_m-v) \cdot\nabla(\phi_m -\phi )\mathrm{dx} + \nonumber \\
+\int_{\Gamma} \gamma(u^m) \left(|\theta_{m}|^{\ell-2} \theta_{m} -|v|^{\ell-2}v\right) (\theta_{m} -v)\mathrm{ds}\geq \nonumber \\ \geq
(L_1)_\# \int_{\Omega}|\nabla (\theta_m-v) |^2\mathrm{dx} +
(L_2)_\#\int_{\Omega}|\nabla (\phi_m -\phi) |^2\mathrm{dx}\geq 0. \label{jmm}
\end{eqnarray}

Let us define 
 \begin{eqnarray*}
\mathcal{J}_m:= \int_{\Omega}
\left( a(\mathbf{u}^m) |\nabla \theta_m|^2 +\sigma (u^m) F (\mathbf{u}^m) \nabla \phi_m\cdot \nabla \theta_m
\right)\mathrm{dx} + \\ +\int_{\Omega}
\left(\sigma (u^m) |\nabla \phi_m|^2 + \sigma (u^m)  \alpha_\mathrm{S} (u^m) \nabla \theta_m\cdot \nabla \phi_m
\right)\mathrm{dx} + \\ 
+\int_{\Gamma}\gamma(u^m) |\theta_m|^\ell\mathrm{ds}.
\end{eqnarray*}

On the one hand, we deduce
 \begin{eqnarray*}
\lim_{m\rightarrow \infty}\mathcal{J}_m\geq
\int_{\Gamma}\gamma(u)\Lambda v\mathrm{ds} +
\int_{\Gamma}\gamma(u)|v|^{\ell-2} v(\theta -v)\mathrm{ds}+ \\ +
\int_{\Omega}a(\mathbf{u})\nabla \theta\cdot \nabla v\mathrm{dx} +
\int_{\Omega}a(\mathbf{u})\nabla v\cdot\nabla (\theta- v)\mathrm{dx} + \\ +
\int_{\Omega}\sigma(u)|\nabla \phi|^2\mathrm{dx} +
\int_{\Omega} \sigma (u) F (\mathbf{u})\nabla \phi \cdot\nabla\theta\mathrm{dx}  + 
\int_{\Omega} \sigma (u)  \alpha_\mathrm{S} (u)\nabla \theta \cdot\nabla\phi\mathrm{dx}.
\end{eqnarray*}

On the other hand,
 taking in (\ref{wvfU})$_m$ the test function $v=\theta^m$, in (\ref{lambda}) the test function $v=\theta$,
 in  (\ref{wvfphiU})$_m$ the test function $w=\phi_m$, and in (\ref{wvfphiU})
the test function $w=\phi$, we deduce
 \begin{eqnarray*}
\lim_{m\rightarrow \infty}\mathcal{J}_m =
\frac{1}{\tau}\int_\Omega f\theta \mathrm{dx} +\int_\Gamma  H\theta \mathrm{ds} -
\frac{1}{\tau}\int_\Omega b(u)\theta \mathrm{dx} + \int_{\Gamma_\mathrm{N}}  g \theta \mathrm{ds}
=\\
=\int_\Omega a( \mathbf{u} ) |\nabla \theta |^2\mathrm{dx}+
\int_\Omega \sigma (u) F( \mathbf{u}) \nabla\phi \cdot\nabla \theta\mathrm{dx} 
+\int_\Gamma \gamma(u) \Lambda \theta \mathrm{ds} + \\ +\int_\Omega \sigma(u)|\nabla\phi |^2\mathrm{dx}
+ \int_\Omega  \sigma(u )  \alpha_\mathrm{S}  (u)\nabla\theta \cdot \nabla \phi\mathrm{dx} .
\end{eqnarray*}

Gathering the above two relations, we find
 \[ 
\int_{\Omega}a(\mathbf{u})|\nabla (\theta- v)|^2\mathrm{dx} + 
\int_{\Gamma}\gamma(u)(\Lambda -|v|^{\ell-2}v) (\theta-v) \mathrm{ds} \geq 0.
\] 
We continue the argument by taking $v=\theta-\delta \varphi$, with $\varphi\in\mathcal{D}(\Gamma)$.
 After dividing by $\delta>0$,
and finally letting $\delta\rightarrow 0^+$ we arrive to
 \begin{eqnarray*}
\int_{\Gamma}\gamma(u)(\Lambda - |\theta|^{\ell-2}\theta )\varphi\mathrm{ds} \geq 0, \quad\forall
\varphi \in\mathcal{D}(\Gamma),
\end{eqnarray*}
which implies that $\Lambda = |\theta|^{\ell-2}\theta$.

 Thus, we are in the condition of  concluding that $(\theta,\phi)$ is the required limit solution, \textit{i.e.}
it solves the limit system (\ref{wvfU})-(\ref{wvfphiU}).
\end{proof}

Thanks to Propositions \ref{propu}, \ref{boundiso} and \ref{contnt}, there exists at least one fixed point of 
$\mathcal{T}$, that is $(\theta,\phi)=\mathcal{T}(\theta,\phi)$, which concludes the
solvability  to (\ref{wvfttm})-(\ref{wvfphim}).

\subsection{Fixed point argument (solvability  to (\ref{wvfttm2})-(\ref{wvfphim2}))}
\label{fpt2}
\ 

Let $\ell\geq 2$, and define an operator $\mathcal{T}$ from $V_{\ell}(\Omega)$ into itself  such that
$\theta=\mathcal{T}(u)$ is the unique solution of Proposition \ref{propt}.

Denote by the well defined continuous operator such that $\mathcal F( u)=\phi$. 
The existence of a unique weak auxiliary solution $\phi$ to the variational equality (\ref{wvfphiU}) is
standard and it can be stated as follows.
\begin{proposition}\label{propphi}
Let $u\in H^1(\Omega)$.
Under the assumptions   (\ref{smm}), (\ref{amax}) and (H5),
the Neumann problem
\begin{equation}\label{wvfphi}
\int_\Omega\sigma(u) \nabla\phi\cdot\nabla w\mathrm{dx}=
\int_\Omega  \sigma(u) \alpha_\mathrm{S} (u)\nabla u\cdot \nabla w \mathrm{dx}+ 
 \int_{\Gamma_\mathrm{N}}  gw \mathrm{ds},
\qquad\forall w\in V,
\end{equation}
admits a unique solution $\phi\in V$.
Moreover, the estimate
\begin{equation}\label{cotaphi}
\|\sqrt{\sigma(u)} \nabla\phi\|_{2,\Omega}\leq \sqrt{\sigma^\#}\alpha^\# \|\nabla u\|_{2,\Omega} 
+\frac{K_2 (P_2+1) }{\sqrt{\sigma_\#}}\|g\|_{2,\Gamma_\mathrm{N}} 
\end{equation}
holds true.
\end{proposition}
\begin{proof} Let us establish the quantitative estimate (\ref{cotaphi}).
We take $w=\phi$ as a test function
 in (\ref{wvfphi}), and we compute by applying the H\"older inequality and (\ref{k2p2})
\begin{eqnarray*} 
\|\sqrt{\sigma(u)}\nabla\phi\|_{2,\Omega} ^2 \leq \left( \alpha^\#
\| \sqrt{\sigma(u)}\nabla u\|_{2,\Omega}  
+\frac{ K_2(P_2+1)}{\sqrt{\sigma_\#}} \|g\|_{2,\Gamma_\mathrm{N}}\right)
\| \sqrt{\sigma(u)}\nabla\phi\|_{2,\Omega} .
\end{eqnarray*}
Then, (\ref{cotaphi}) arises.
\end{proof}

\begin{proposition}\label{propt}
Let $\mathbf{u} =(u,\phi) \in \left( H^1(\Omega) \right)^2$. 
Under the assumptions   (\ref{smm}), (\ref{gamm}) and (\ref{isof})-(\ref{bmm}),
there exists a unique solution $\theta\in   V_\ell(\Omega)$ to the  power type elliptic problem 
\begin{eqnarray}
\frac{1}{\tau}\int_\Omega b(u)\theta v\mathrm{dx}+
\int_\Omega a (\mathbf{u} )\nabla \theta \cdot\nabla v\mathrm{dx}+
\int_\Omega \sigma(u) F ( \mathbf{u} ) \nabla\phi \cdot\nabla v\mathrm{dx} +\nonumber \\
+\int_\Gamma \gamma(u) |\theta |^{\ell-2}\theta v \mathrm{ds} \label{wvftt} = 
\frac{1}{\tau}\int_\Omega fv \mathrm{dx} +\int_\Gamma  Hv \mathrm{ds} ,
\end{eqnarray}
for all $v\in V_{\ell}(\Omega)$. If  $\phi\in V$ satisfies (\ref{cotaphi}), then the following estimate
\begin{eqnarray}
\frac{b_\#}{2\tau }\| \theta \|_{2,\Omega}^2 +
\frac{ a_\#}{2} \|\nabla \theta \|_{2,\Omega}^2
+\frac{\gamma_\#}{\ell '} \| \theta \|_{\ell,\Gamma} ^\ell \leq 
\frac{ 1 }{ 2\tau b_\# }\|f\|_{2,\Omega}^2 +
\frac{1}{\ell '\gamma_\#^{1/(\ell-1)}}
\| H\|_{\ell ',\Gamma}^{\ell '}+ \nonumber \\  +\frac{(F^\#)^2\sigma^\#}{ a_\#}
\left( \sigma^\#(\alpha^\# )^2\|\nabla u \|_{2,\Omega}^2 +
\frac{K_2 ^2(P_2+1)^2 }{\sigma_\#}\|g\|_{2,\Gamma_\mathrm{N}}^2\right) . \label{cotatt}
\end{eqnarray}
holds true.
\end{proposition}
\begin{proof}
Taking $v=\theta$ as a test function in (\ref{wvftt}), and applying (\ref{bmm}), (\ref{amm}), 
(\ref{gamm}), (\ref{isof}), and the H\"older inequality, we find
\begin{eqnarray*}
\frac{b_\#}{\tau}\|\theta\|_{2,\Omega}^2+ a_\# \|\nabla \theta\|_{2,\Omega}^2 
+\gamma_\# \| \theta\|_{\ell,\Gamma} ^\ell \leq  \\ \leq
\frac{ 1 }{\tau} \|f\|_{2,\Omega}\|\theta\|_{2,\Omega}+
F^\#\| \sqrt{\sigma(u)} \nabla\phi \|_{2,\Omega} \|\sqrt{\sigma(u)} \nabla \theta\|_{2,\Omega}+
\|H\|_{\ell ',\Gamma}\|\theta\|_{\ell,\Gamma} .
\end{eqnarray*}
Applying (\ref{cotaphi}) and the Young inequality,  we compute
\begin{eqnarray*}
\|\sqrt{\sigma(u)} \nabla\phi\|_{2,\Omega}\|\sqrt{\sigma(u)} \nabla \theta\|_{2,\Omega}
\leq \frac{a_\#}{2} \|\nabla \theta\|_{2,\Omega}^2 + \\
+\frac{\sigma^\# }{a_\#} \left(\sigma^\# (\alpha^\#)^2 \|\nabla u\|_{2,\Omega} ^2
+\frac{K_2^2(P_2+1) ^2}{\sigma_\#}\|g\|_{2,\Gamma_\mathrm{N}}^2 \right).
\end{eqnarray*}
Then, arguing as in (\ref{eqrbum}), we deduce (\ref{cotatt}).
\end{proof}

Next, let us determine whose radius make possible that  the operator $\mathcal{T}$ maps a closed ball into itself.
\begin{proposition}\label{bound}
Let (\ref{asfg}) be fulfilled.
For $\tau\leq a_\# / b_\#$ and  $R>0$ being defined as in (\ref{defr}),
the operator $\mathcal{T}$ verifies $\mathcal{T}(\overline{B_R})\subset \overline{B_R}$, with
 $B_R$ denoting the open ball of $H^1(\Omega)$ with radius $R$.
\end{proposition}
\begin{proof}
Let  $u\in H^1(\Omega)$ and $\theta=\mathcal{T}(u)$ be the unique solution according to Proposition \ref{propt}.
Considering (\ref{cotatt}) and $\min\left\lbrace b_\#/\tau ,a_\#\right\rbrace=a_\#$,
 the proof is complete by defining $R$ such that
\begin{eqnarray}
R\left(\sqrt{ a_\# }-2
\frac{F^\# \sigma^\#\alpha^\# }{ \sqrt{a_\#}} \right) &=&\left(
\frac{ 1 }{\tau b_\# }\| f \|_{2,\Omega}^2 + 
\frac{2}{\ell '\gamma_\#^{1/(\ell-1)}} \| H \|_{\ell ',\Gamma}^{\ell '} \right)^{1/2}+\nonumber \\
&&+ 2
F^\# K_2(P_2+1) \sqrt{\frac{\sigma^\#}{  a_\#\sigma_\#}}\|g\|_{2,\Gamma_\mathrm{N}} , \quad \label{defr}
\end{eqnarray}
taking the assumption (\ref{asfg}) be account.
\end{proof}

\begin{proposition}\label{contnt2}
Let $\lbrace u_m\rbrace_{m\in\mathbb{N}}$ be a sequence such that
weakly converges to $u$ in $ H^1(\Omega)$, then
 the solution $(\theta_m,\phi_m)$ according to Propositions \ref{propphi} and \ref{propt} weakly converges in 
 $V_\ell (\Omega)\times H^1(\Omega)$, and its limit is a solution according to Propositions \ref{propphi} and \ref{propt}
\end{proposition}
\begin{proof} Let
$(\theta_m,\phi_m)$ be  the solution  according to Propositions \ref{propphi} and \ref{propt}
 and corresponding to $u_m$ for each $m\in\mathbb{N}$.
The estimates (\ref{cotaphi}) and (\ref{cotatt}) guarantee that the sequence $(\theta_m,\phi_m)$  is uniformly
 bounded in $V_\ell (\Omega)\times H^1(\Omega)$.  Thus, we can extract  a subsequence of $(\theta_m,\phi_m)$
still denoted by $(\theta_m,\phi_m)$, weakly convergent to $(\theta,\phi)$ in $V_\ell (\Omega)\times H^1(\Omega)$.
 Similar arguments in the proof of Proposition \ref{contnt} the weak limit $(\theta,\phi)$
 solves the variational system consisting of
(\ref{wvfphi}) and (\ref{wvftt}), which concludes the proof of Proposition \ref{contnt2}.
\end{proof}

Thanks to Propositions \ref{propphi}, \ref{propt}, \ref{bound} and \ref{contnt2}, there exists at least one fixed point of 
\[
\mathcal{T}:u\mapsto (u, \mathcal{F}(u))\mapsto \theta,
\] that is $\theta=\mathcal{T}(\theta)$ and $\phi=\mathcal{F}(\theta)$, which concludes the
solvability  to (\ref{wvfttm2})-(\ref{wvfphim2}).

\section{Time discretization technique}
\label{tdt}

In this section, we apply the method of discretization in time
\cite{kacur,plusweber,karel}.

We decompose the time interval $I=[0,T]$ into $M$ subintervals $I_{m,M}$ of size
$\tau$ such that $M=T/\tau \in\mathbb N$,  \textit{i.e.}
$I_{m,M}=[(m-1)T/M,mT/M]$ for  $m\in\{1,\cdot\cdot\cdot,M\}$.
We set $t_{m,M}=mT/M$.
 Thus, the problem (\ref{wvftta}) is approximated by the following recurrent sequence of time-discretized problems
\begin{eqnarray}
\frac{1}{\tau}\int_\Omega B (\theta^m) v\mathrm{dx}+\int_\Omega
 a (\theta^m,\phi^m )\nabla \theta^{m}\cdot\nabla v\mathrm{dx}
+\int_\Gamma \gamma(\theta^m)|\theta^m|^{\ell-2} \theta^{m} v \mathrm{ds} +\nonumber \\ +
\int_\Omega \sigma(\theta^m)  F( \theta^m,\phi^m ) \nabla\phi^m \cdot\nabla v\mathrm{dx} 
 = 
\frac{1}{\tau}\int_\Omega B (\theta^{m-1} )v\mathrm{dx}
+\int_\Gamma  h(t_{m,M}) v  \mathrm{ds}  ,\label{td2}
\end{eqnarray}
for all $v\in V_{\ell}(\Omega)$, and the problem (\ref{wvfpha}) is approximated by either
 (\ref{wvfphim}) or (\ref{wvfphim2}) for all $w\in V$, corresponding to the two different approaches.
The existence of weak solutions pair $(\theta^m,\phi^m) \in V_{\ell}(\Omega)\times V$ to the above systems of elliptic problems  is established in Section  \ref{sfpt}
with $H =h(t_{m,M})$.

Since $\theta^0\in L^2(\Omega)$ is known, we determine $(\theta^{1},\phi^1)$
 as the unique solution
of   the Neumann-power type elliptic problems (\ref{wvfttm})-(\ref{wvfphim}) or
(\ref{wvfttm2})-(\ref{wvfphim2}),
 and we inductively proceed. 

Denote  by 
$\{  {\theta}_M\}_{M\in\mathbb N}$, $\{  {\phi}_M\}_{M\in\mathbb N}$
 and $\{  Z_M\}_{M\in\mathbb N}$ 
 the sequences of the (piecewise constant in time) functions,
$   {\theta}_M : [0,T] \rightarrow   V_\ell(\Omega)$,
$   {\phi}_M : ]0,T] \rightarrow   V$ 
and  $  Z_M:[0,T]\rightarrow L^2(\Omega)$, defined by, respectively,  a.e. in $\Omega$
\begin{eqnarray}\label{deftm}
   {\theta}_M (t) &:=& \left\{\begin{array}{ll}
\theta^{0}&\mbox{ for }t=0\\
 \theta^{m}&\mbox{ for } t\in ]t_{m-1,M},t_{m,M}]
\end{array}\right.  \\
    {\phi}_M (t) &:=& \phi^m \quad\mbox{ for all }t\in ]t_{m-1,M},t_{m,M} ],\label{defpm}
\end{eqnarray}
in accordance with one of the two variational formulations  (\ref{wvfphim}) and (\ref{wvfphim2}),
while
\begin{equation}\label{defzm}
  Z_M (t):=\left\{\begin{array}{ll}
B(\theta^{0} ) &\mbox{ for }t=0\\
Z^{m}&\mbox{ for } t\in ]t_{m-1,M},t_{m,M}]
\end{array}\right. \mbox{ in }\Omega,
\end{equation}
with 
 the discrete derivative with respect to $t$ at the time $t=t_{m,M}$:
\[
Z^{m}:=\frac{B(\theta^{m})-B(\theta^{m-1}) }{\tau}.
\]

While $\theta_M$ is the Rothe function obtained from $\theta^m$ by piecewise constant interpolation 
with respect to time $t$,
 the Rothe function, obtained from $\theta^m$ by piecewise linear interpolation with respect to time $t$, 
  $\Theta_M$ is
  \[
\Theta_M(\cdot,t)=\theta^{m-1}+(t-t_{m-1,M}) \frac{\theta^m- \theta^{m-1}}{  \tau }.
  \]
For our purposes, we introduce the following definition.
\begin{definition} \label{defro}
We say that  $\{ \widetilde{B}_M= \widetilde{B}(\theta_M)\}_{M\in\mathbb N}$ is the Rothe sequence (affine
on each time interval) if
\[
 \widetilde{B}(\cdot,\theta_M (t)) = B(\cdot,\theta^{m-1})  + \frac{ t-t_{m-1,M}}{\tau}
\left( B(\cdot,\theta^m) - B(\cdot,\theta^{m-1}) \right)
  \]
in  $\Omega$, for all $t\in I_{m, M}$,
for all  $m\in \{1,\cdots,M\}$.
  \end{definition} 
  
 Denoting $h_M(t)=h(t_{m,M})$ for $t\in ]t_{m-1,M},t_{m,M}]$ and $m\in\{1,\cdots,M\}$, the triple
  $ ({\theta}_M ,\phi_M,{Z}_M )$ solve
 \begin{eqnarray}
\int_0^T\int_\Omega  {Z}_M v\mathrm{dx}\mathrm{dt} +
\int_{Q_T}  a ( {\theta}_M ,   {\phi}_M  )\nabla  {\theta}_M \cdot\nabla v\mathrm{dx}\mathrm{dt}
+\int_{\Sigma_T} \gamma( {\theta}_M )| {\theta}_M |^{\ell-2}  {\theta}_M  v \mathrm{ds}\mathrm{dt}  +\nonumber \\
  +
\int_{Q_T} \sigma ( \theta_M ) F(  \theta_M ,   \phi_M  ) \nabla   {\phi}_M  \cdot\nabla v\mathrm{dx}\mathrm{dt}
= 
\int_{\Sigma_T}  h_M v  \mathrm{ds}\mathrm{dt} .\qquad \label{tdM}
\end{eqnarray}

\subsection{Proof of Theorem \ref{abst1}}
\label{sabst1}
\

Here,  $   {\phi}_M $ solves
\begin{equation}\label{pbphiM}
\int_\Omega\sigma (\theta_M)\nabla\phi_M\cdot\nabla w\mathrm{dx}
 + \int_\Omega \sigma (\theta_M) \alpha_\mathrm{S}(\theta_M)
\nabla\theta_M \cdot \nabla w \mathrm{dx}=\int_{\Gamma_\mathrm{N}}  g w \mathrm{ds} ,
\end{equation}
for $w\in V$ and a.e. in $]0,T[$.
  
We begin by establishing the uniform  estimates to $   {\theta}_M $ and $   {\phi}_M $.
\begin{proposition}\label{tmzm}
Let $   {\theta}_M $ and $   {\phi}_M $  be the (piecewise constant in time) functions defined in
 (\ref{deftm})-(\ref{defpm}).
Then the following estimate holds:
\begin{eqnarray}\label{cotaul}
\max_{1\leq m\leq M} \int_\Omega \Psi( \theta^m )\mathrm{dx}+( L_1)_\# \| \nabla \theta_M\|_{2,Q_T} ^2 
+\frac{(L_{2})_\# }{2}\| \nabla \phi_M\|_{2,Q_T} ^2 +
\nonumber\\
+\frac{\gamma_\#}{\ell '} \|  \theta_M\|_{\ell,\Sigma_T } ^\ell 
\leq b^\#  \| \theta^0\|_{2,\Omega} ^2 +
\frac{1}{\ell '\gamma_{\#}^{1/(\ell-1)} } \| h \|_{\ell',\Sigma_T }^{\ell '} +
T \frac{K_2^2(P_2+1)^2}{2 (L_{2})_\#} \|g\|_{2,\Gamma_\mathrm{N}}^2 .
 \end{eqnarray}
\end{proposition}
\begin{proof}
 Let $m\in \{1,\cdots,M\}$ be arbitrary. 
Choosing $v=\theta^m\in V_{\ell}(\Omega)$ and $w=\phi^m \in V$ as  test functions
 in (\ref{wvfttm})-(\ref{wvfphim}), we sum the obtained relations,
and 
arguing as in (\ref{eql1l2})-(\ref{eqrbum}), we have
\begin{eqnarray}
\frac{1}{\tau}
\int_\Omega (B(\theta^m)-B(\theta^{m-1}) )\theta^m\mathrm{dx} + 
( L_1)_\# \|\nabla \theta^{m} \|_{2,\Omega}^2 +
 \frac{(L_{2})_\#}{2}  \|\nabla\phi ^{m} \|_{2,\Omega}^2 +\nonumber \\ 
+
\int_\Gamma \gamma(\theta^{m})|\theta^{m}|^{\ell}  \mathrm{ds}
\leq \mathcal{R}(0, \| h(t_{m,M}) \|_{\ell ',\Gamma}^{\ell '} )
+\frac{1}{\ell} 
\int_\Gamma \gamma(\theta^{m})|\theta^{m}|^{\ell}  \mathrm{ds}, \label{eqbb}
\end{eqnarray}
with $\mathcal{R}$ being the increasing continuous function defined in (\ref{cotatphi}).
By  (\ref{bpsi}), we have
\begin{eqnarray*}
\sum_{i=1}^m\int_\Omega (B(\theta^i)-B(\theta^{i-1}) )\theta^i\mathrm{dx} \geq
\int_\Omega (\Psi(\theta^m) - \Psi(\theta^{0})  )\mathrm{dx} .
\end{eqnarray*}

Therefore, summing over $i=1,\cdots, m$
into (\ref{eqbb}), multiplying by $\tau$, and  inserting the previous inequality, we obtain
\begin{eqnarray}
 \int_\Omega \Psi( \theta^m )\mathrm{dx}& +&
\tau\sum_{i=1}^m \left( ( L_1)_\# \|\nabla \theta^{i} \|_{2,\Omega}^2
+\frac{1}{\ell '} 
\int_\Gamma \gamma(\theta^{i})|\theta^{i}|^{\ell}  \mathrm{ds} +
 \frac{(L_{2})_\#}{2} \|\nabla\phi ^{i} \|_{2,\Omega}^2 \right) \nonumber \\ &\leq  &
 \int_\Omega \Psi( \theta^0 )\mathrm{dx} +  
 \tau\sum_{i=1}^m\mathcal{R}(0,  \| h(t_{m,M}) \|_{\ell ',\Gamma}^{\ell '} ) . \label{cotat2}
\end{eqnarray}
Therefore, we find the uniform estimate (\ref{cotaul}) by taking the maximum over $m\in\{1,\cdots,M\}$ in the previous estimate
and applying Lemma \ref{lbmm} provided by (\ref{bmm}).
\end{proof}

A direct application of Proposition \ref{tmzm} ensures the following proposition.
\begin{proposition}\label{ttphi}
There exist $\theta,\phi: Q_T\rightarrow\mathbb{R}$ and subsequences of  $(\theta_M,\phi_M)$,
still labelled by  $(\theta_M,\phi_M)$,  such that
\begin{eqnarray}
  {\theta}_M \rightharpoonup \theta &\mbox{ in }& V_\ell(Q_T); \\
  \phi_M \rightharpoonup \phi& \mbox{ in }& L^2(0,T;V),
\end{eqnarray}
 as $M$ tends to infinity.
 Moreover, there exists $Z: Q_T\rightarrow\mathbb{R}$ such that
\[
\partial_t\widetilde{B}(\theta_M)\rightharpoonup Z \quad \mbox{in } L^{\ell ' }(0,T;(V_\ell(\Omega))'),
\]
 as $M$ tends to infinity.
\end{proposition}
\begin{proof}
Considering the uniform  estimates to $   {\theta}_M $ and $   {\phi}_M $  that are established in
Proposition \ref{tmzm}, we  extract  subsequences, still denoted by  $   {\theta}_M $  and $   {\phi}_M $,
weakly convergent in $  V_\ell (Q_T )$ and $ L^2(0,T;V)$, respectively, to $\theta$ and $\phi$.

Let  $p=\max\{\ell,2\}=\ell$. Note that $L^p(0,T;V_\ell(\Omega))\hookrightarrow V_\ell(Q_T)$.
By definition of norm, we find
\[ 
\|\partial_t\widetilde{B}(\theta_M)\|_{L^{p'}(0,T;(V_\ell(\Omega))')} = \sum_{m=1}^M
\int_{(m-1)\tau}^{m\tau} 
\sup_{v\in L^p(0,T;V_\ell(\Omega)) \atop \|v\|\leq 1} \langle Z^m,v\rangle  \mathrm{dt} 
\leq C,
\] 
with $C>0$ being a constant independent on $M$, by estimating in (\ref{td2})
the term involving $Z^m$
by means of the rest terms using  the uniform  estimates established in (\ref{cotat2}).
Hence, we can extract a subsequence, still denoted by $\partial_t\widetilde{B}(\theta_M)$, weakly
convergent to $Z$ in  $L^{p' }(0,T;(V_\ell(\Omega))')$.
\end{proof}

In the following proposition, we state the strong convergence  of $B(\theta_M)$ and of $\theta_M$.
\begin{proposition}\label{pbttm}
Under  (\ref{smm})-(\ref{gamm}) and  (\ref{isof})-(\ref{amm}), the solution $\theta^m$ of (\ref{td2}) satisfies
\begin{eqnarray}
b_\#\|\theta^{m }-\theta^{m-1 }\|_{2,\Omega}  ^2\leq
\int_\Omega (B(\theta^{m })-B(\theta^{m-1 })) (\theta^{m }-\theta^{m-1 }) \mathrm{dx} \leq \nonumber \\
\leq C\left(
 \|\theta^{m }-\theta^{m-1 }\|_{\ell,\Gamma}  \tau^{1/\ell } +  \| \theta^{m }-\theta^{m-1 }\|_{2,\Omega}\sqrt{\tau} \right),\label{bbj}
\end{eqnarray}
with  $C$ being a positive constant. Moreover,
for a subsequence, there hold 
\begin{eqnarray}
B(  {\theta}_M) \rightarrow B(\theta) &\mbox{ in }& L^1(Q_T); \\
 \theta_M \rightarrow \theta & \mbox{ a.e. in }& Q_T,
\end{eqnarray}
 as $M$ tends to infinity.
\end{proposition}
\begin{proof}
Let $k\in\mathbb{N}$. Let us sum up (\ref{td2}) for $m=j+1,\cdots,j+k$ and multiply by $\tau$, obtaining
\[
\int_\Omega (B(\theta^{j+k })-B(\theta^{j })) v \mathrm{dx} \leq 
\mathcal {I }_\Gamma^j +\mathcal{I}_\Omega^j ,
\] 
where
\begin{eqnarray*}
\mathcal{I}_\Gamma^j&:=&  \tau\sum_{m=j+1}^{j+k} \int_\Gamma | (\gamma(\theta^m)|\theta^{m}|^{\ell-2} \theta^{m} -
 h(t_{m,M}) )v| \mathrm{dx} \\
 &\leq & \int_{j\tau}^{(j+k)\tau}
\| \gamma(\theta_M)|\theta_M|^{\ell-2} \theta_M -
 h  \|_{\ell ',\Gamma } \|v \|_{\ell,\Gamma}\mathrm{dt} ; \\
\mathcal{I}_\Omega^j &:=& \tau \sum_{m=j+1}^{j+k}\int_\Omega |( a ( \theta^m,\phi^m )\nabla \theta^{m} +
  F ( \theta^m,\phi^m) \nabla\phi^m )v|  \mathrm{dx} \\
&\leq & \int_{j\tau}^{(j+k)\tau}
\| a ( \theta_M,\phi_M )\nabla \theta_M +
  F ( \theta_M,\phi_M) \nabla\phi_M \|_{2,\Omega}   
\| v\|_{2,\Omega} \mathrm{dt} . 
\end{eqnarray*}
Here, we used the H\"older inequality and the definition of $ \theta_M$ and of $\phi_M$.

Let us compute $\mathcal {I }_\Gamma^j $ and $\mathcal{I}_\Omega^j$ by  applying  the estimate (\ref{cotaul}).
Using the assumption (\ref{gamm}) and after the H\"older inequality, we deduce
\begin{equation}\label{ij1}
\mathcal{I}_\Gamma^j
 \leq \|v \|_{\ell,\Gamma} \int_{j\tau}^{(j+k)\tau}\left(
\gamma^\#\|\theta_M\|_{\ell ,\Gamma }^{\ell-1}+\| h  \|_{\ell ',\Gamma } \right)\mathrm{dt} 
\leq  \|v \|_{\ell,\Gamma} C (k\tau)^{1/\ell }.
 \end{equation}
 Using the assumptions (\ref{smm}), (\ref{isof}) and (\ref{amm}),  and after the H\"older inequality,  we deduce
\begin{equation}\label{ij2}
\mathcal{I}_\Omega^j \leq \| v\|_{2,\Omega} \int_{j\tau}^{(j+k)\tau}
\left(a^\#\| \nabla \theta_M\|_{2,\Omega}  + \sigma^\#
  F ^\#\| \nabla\phi_M \|_{2,\Omega}   \right) \mathrm{dt} \leq \| v\|_{2,\Omega}C\sqrt{k\tau}. 
\end{equation}

Hence, we find 
\begin{equation}\label{bbjk}
\int_\Omega (B(\theta^{j+k })-B(\theta^{j })) v \mathrm{dx} \leq 
 \|v \|_{\ell,\Gamma} C (k\tau)^{1/\ell } +  \| v\|_{2,\Omega}C\sqrt{k\tau} .
\end{equation}
In particular,
 the  estimate  (\ref{bbj}) follows by taking $j=m-1$, $k=1$ and $v=\theta^{m }-\theta^{m-1 }$, and applying
Lemma \ref{lbmm}.

To prove the convergences, we will apply Lemma \ref{lbmm2}.
Considering the weak convergence of $\theta_M$ established in Proposition \ref{ttphi} and the estimate (\ref{cotaul}),
in order to apply Lemma \ref{lbmm2} it remains to prove that the condition (\ref{cbb}) is fulfilled.

Let $0<z<T$ be arbitrary. Since the objective is to find convergences, it suffices to take $M>T/z$, which means $\tau<z$.
Thus, there exists $k\in\mathbb{N}$ such that $k\tau<z\leq (k+1)\tau$. Moreover, we may choose $M>k+1$ deducing
\begin{eqnarray*}
\int_0^{T-z}\int_\Omega (B(\theta_M(t+z))-B(\theta_M(t))) (\theta_M(t+z)-\theta_M(t))\mathrm{dx} \mathrm{dt}\leq \\
\leq \sum_{j=1 }^{M-k} \int_{(j-1)\tau}^{(j+k)\tau}
\int_\Omega (B(\theta^{j+k })-B(\theta^{j })) (\theta^{j+k }-\theta^{j })\mathrm{dx} .
\end{eqnarray*}

Taking $v=\theta^{j+k }-\theta^j$ in (\ref{bbjk}) and then summing up
for $j=1,\cdots, M-k$,   we find
\begin{eqnarray*}
\int_0^{T-z}\int_\Omega (B(\theta_M(t+z))-B(\theta_M(t))) (\theta_M(t+z)-\theta_M(t))\mathrm{dx} \mathrm{dt}\leq \\
\leq (k+\tau)
\sum_{j=1 }^{M-k}
\int_\Omega (B(\theta^{j+k })-B(\theta^{j })) (\theta^{j+k }-\theta^{j })\mathrm{dx} \leq \\ \leq
\sum_{j=1 }^{M-k} \int_{(j-1)\tau}^{j\tau +k }
 \left(  \|\theta^{j+k }-\theta^{j }\|_{\ell,\Gamma} C (k\tau)^{1/\ell } + 
  \| \theta^{j+k }-\theta^{j }\|_{2,\Omega}C\sqrt{k\tau} \right) \mathrm{dt} .
\end{eqnarray*}
Arguing as in (\ref{ij1}) and (\ref{ij2}), we conclude
\begin{eqnarray*}
\int_0^{T-z}\int_\Omega (B(\theta_M(t+z))-B(\theta_M(t))) (\theta_M(t+z)-\theta_M(t))\mathrm{dx} \mathrm{dt}\leq \\ \leq
 C\left( (k\tau)^{1/\ell }(k\tau+\tau)^{ 1/\ell'}  +  (k\tau)^{1/2} (k\tau+\tau )^{1/2}\right)=C 
 \left( 2^{ 1/\ell'}+ 2^{ 1/2}\right) z.
\end{eqnarray*}
Thus,
all hypothesis of  Lemma \ref{lbmm2} are fulfilled.
Therefore, Lemma \ref{lbmm2} assures that $B(\theta_M)$ strongly converges to $B(\theta)$ in $L^1(Q_T)$. Consequently, up to a subsequence,
$B(\theta_M)$  converges to $B(\theta)$ a.e. in $Q_T$. Since $B$ is strictly monotone, 
$\theta_M$  converges to $\theta$ a.e. in $Q_T$ (see, for instance, \cite{lions}).
\end{proof}

  Now we are able to identify the limit $Z$.
\begin{proposition}\label{duuw}
The limit $Z$ satisfies 
\[
Z=\partial_t (B(\theta)) \quad\mbox{ in } L^{\ell ' }(0,T;(V_\ell(\Omega))').
\]
\end{proposition}
\begin{proof} 
 For a fixed $t$, there exists
$m\in\{1,\cdots,M\}$ such that  $t\in]t_{m-1,M},t_{m,M}]$.
From definition \ref{defzm}  we have
\begin{eqnarray*}
\int^t_0  {Z}_M(\varsigma)\mathrm{d\varsigma}=\sum_{j=1}^{m-1}\int_{(j-1)\tau}^{j\tau}
\frac{B(\theta^{j})(\varsigma)-B(\theta^{j-1})(\varsigma)}{ \tau}\mathrm{d\varsigma} + \\
+\int_{(m-1)\tau}^{t}
\frac{ B(\theta^m) (\varsigma)- B(\theta^{m-1})(\varsigma) }{\tau}\mathrm{d\varsigma} =\\
=B(\theta^{m-1})- B(\theta^{0}) + \frac{ t-(m-1)\tau }{\tau}
\left( B(\theta^m) - B(\theta^{m-1}) \right) = \widetilde{B}(\theta_M(t)) - B(\theta^{0}) 
\end{eqnarray*}
 in $\Omega$.
By the Riesz theorem, the  bounded linear functional 
$v\in L^2(\Omega)\mapsto \int^t_0  ({Z}_M(\varsigma),v)\mathrm{d\varsigma}$ is
 (uniquely)  representable by the element $\widetilde{B}(\theta_M(t)) - B(\theta^{0}) $ from $L^2(\Omega).$
Using the corresponding definitions we compute
\begin{eqnarray*}
\int_0^T\|\widetilde{B}(\theta_M(t)) - B(\theta_M) \|_{2,\Omega}^2\mathrm{dt} =
\sum_{m=1}^M \|Z^m \|_{2,\Omega}^2\int_{(m-1)\tau}^{m\tau} (t-m\tau)^2\mathrm{dt}= \\
=\frac{\tau^3}{3}\sum_{m=1}^M \|Z^m \|_{2,\Omega}^2 =
\frac{\tau}{3}
\sum_{m=1}^M\|B(\theta^m) - B(\theta^{m-1}) \|_{2,\Omega}^2\leq \\
\leq (b^\#)^2 \frac{\tau}{3}
\sum_{m=1}^M\|\theta^m - \theta^{m-1} \|_{2,\Omega}^2.
\end{eqnarray*}
Applying (\ref{bbj}) and Proposition \ref{tmzm} we have
\[
\int_0^T\|\widetilde{B}(\theta_M(t)) - B(\theta_M) \|_{2,\Omega}^2\mathrm{dt}\leq C 
\left( \tau^{1/\ell +1/\ell'} +\tau^{1/2+1/2} \right) = C\tau,
\]
and consequently $\widetilde{B}(\theta_M)$ converges to $ B(\theta) $. 
By the uniqueness of limit, we deduce 
\[
\int^t_0  Z(\varsigma)\mathrm{d\varsigma}=B(\theta) - B(\theta^{0}),
\]
which concludes the proof.
\end{proof}

We emphasize that the above convergences are sufficient to identify the limit $\phi$ as stated in the following proposition,
but they are not sufficient to identify the temperature $\theta$ as a solution, because on the one hand the apparent
nonlinearity of the coefficients destroy the weak convergence, 
on the other hand, the weak-weak convergence does not imply weak convergence.
\begin{corollary}\label{pphi}
Let  $(\theta,\phi)$ be in accordance with Propositions \ref{ttphi} and \ref{pbttm}, then they verify (\ref{wvfpha}).
\end{corollary}
\begin{proof}
Let $(\theta_M,\phi_M)$ solve (\ref{tdM})-(\ref{pbphiM}). Applying Propositions \ref{ttphi} and \ref{pbttm}, and
 the Krasnoselski theorem to the Nemytskii operators $\sigma$ and $\alpha_\mathrm{S}$, we have
\begin{eqnarray}\label{phim1}
 \sigma ( \theta_M )\nabla \phi_M\rightharpoonup \sigma ( \theta )\nabla \phi &\mbox{ in }& \mathbf{L}^2(Q_T); \\
  \sigma ( \theta_M ) \alpha_\mathrm{S}(  \theta_M ) \nabla \theta_M\rightharpoonup 
   \sigma ( \theta ) \alpha_\mathrm{S}(  \theta ) \nabla \theta &\mbox{ in }& \mathbf{L}^2(Q_T) \quad
\mbox{ as }   M\rightarrow +\infty .\label{phim2}
\end{eqnarray}
Thus,   we may pass to the limit 
in (\ref{pbphiM}) as $M$ tends to infinity, concluding that $(\theta,\phi)$ verifies (\ref{wvfpha}).
\end{proof}

\subsection{Proof of Theorem \ref{abst2}}
\label{sabst2}
\

This proof follows \textit{mutatis mutandis} the structure of the proof of Theorem \ref{abst1} (cf. Subsection 
\ref{sabst1}). We only sketch its main steps.
\begin{enumerate}
\item The uniform  estimates to $   {\theta}_M $ and $   {\phi}_M $ are as follows.
The quantitative estimate (\ref{cotaul}) reads
\begin{eqnarray*}
 \int_\Omega \Psi( \theta^m )\mathrm{dx}+ a_\# \| \nabla \theta_M\|_{2,Q_T} ^2 
+\frac{\gamma_\#}{\ell '} \|  \theta_M\|_{\ell,\Sigma_T } ^\ell \leq
\nonumber\\
\leq b^\#  \| \theta^0\|_{2,\Omega} ^2 +
\frac{1}{\ell '\gamma_{\#}^{1/(\ell-1)} } \| h \|_{\ell',\Sigma_T }^{\ell '} +
T \frac{K_2^2(P_2+1)^2}{2 (L_{2})_\#} \|g\|_{2,\Gamma_\mathrm{N}}^2 ,
 \end{eqnarray*}
by using the argument to estimate  (\ref{cotatt}). Namely, (\ref{eqbb}) reads
\begin{eqnarray*}
\frac{1}{\tau}
\int_\Omega (B(\theta^m)-B(\theta^{m-1}) )\theta^m\mathrm{dx} + 
a_\# \|\nabla \theta^{m} \|_{2,\Omega}^2 
+
\int_\Gamma \gamma(\theta^{m})|\theta^{m}|^{\ell}  \mathrm{ds} \leq \\
\leq \mathcal{R}(0, \| h(t_{m,M}) \|_{\ell ',\Gamma}^{\ell '} )
+\frac{1}{\ell} 
\int_\Gamma \gamma(\theta^{m})|\theta^{m}|^{\ell}  \mathrm{ds},
\end{eqnarray*}
where $\mathcal{R}$ is the increasing continuous function defined in (\ref{cotatphi}), with
 $(L_2)_\#=a_\#\sigma_\#/ (2 F^\#\sigma^\#)$. In addition, summing the quantitative estimate 
 \[
 \sqrt{\sigma_\#} \|\nabla\phi ^{m} \|_{2,\Omega} \leq\sqrt{\sigma_\#}a^\# \|\nabla\theta ^{m-1} \|_{2,\Omega}+
  \frac{K_2(P_2+1)}{  \sqrt{\sigma_\#} } \|g\|_{2,\Gamma_\mathrm{N}},
 \]
it results in
 \[
 \| \nabla \phi_M\|_{2,Q_T}  \leq a^\# \|\nabla\theta_M \|_{2,Q_T}+ T
  \frac{K_2(P_2+1)}{  \sigma_\#} \|g\|_{2,\Gamma_\mathrm{N}} .
\]
\item For subsequences of  $   {\theta}_M $ and $   {\phi}_M $,
the weak convergences hold according to Propositions \ref{ttphi} and \ref{pbttm}, which guarantee the required result.
\end{enumerate}

\section{Existence of solutions to the TE problem}

The objective is the passage to the limit in the abstract boundary value problems
introduced in Section \ref{sabst} as the time step goes to zero ($M\rightarrow +\infty$),
with the coefficients being  defined by
\begin{eqnarray*}
b(\cdot, v) &=& \rho(\cdot,v)c_\mathrm{v}(\cdot, v);\\
a(\cdot, v,w)&=&k(\cdot, v) + T_\mathcal{M}(w)\alpha_\mathrm{S}(\cdot, v)\sigma (\cdot, v);\\
F(\cdot, v,w)&=&\Pi(\cdot, v)+T_\mathcal{M}(w),
\end{eqnarray*}
where $T_\mathcal{M}$ is the $\mathcal{M}$-truncation function defined by $T_\mathcal{M}(z)=
\max(-\mathcal{M},\min(\mathcal{M},z))$. By the definition of truncated functions, we choose 
\begin{eqnarray*}
a_\# &=& k_\#-\mathcal{M}\alpha^\#\sigma^\#;\\
F^\#  &=& \Pi^\#+\mathcal{M}.
\end{eqnarray*}
taking (\ref{akM}) into account. Under these choices, the assumptions (\ref{sss1}), (\ref{sss2}) and (\ref{sss3})
imply (\ref{afg}), (\ref{sfg}) and (\ref{asfg}), respectively. 

Let us foccus the present proof in accordance with the approximated solutions that are established in Theorem
\ref{abst1}. Analogous argument is valid for 
the approximated solutions that are established in Theorem \ref{abst2}.

Let us redefine the electrical current density  as
\[ 
\mathbf{j}( \theta , \phi)= \sigma ( \theta )\nabla \phi + 
   \sigma ( \theta ) \alpha_\mathrm{S}(  \theta ) \nabla \theta .
\] 
Analogously for $\mathbf{j}_M= \mathbf{j}(\theta_M ,\phi_M)$  or simply $\mathbf{j}_M$ and
 $\mathbf{j}$ whenever the meaning is not ambiguous.
 
Let  $ (\theta_M,  {\phi}_M )$ solve (\ref{tdM})-(\ref{pbphiM}), which may be rewritten as
 \begin{eqnarray}
\int_0^T\int_\Omega  {Z}_M v\mathrm{dx}\mathrm{dt} +
\int_{Q_T}  k ( {\theta}_M  )\nabla  {\theta}_M \cdot\nabla v\mathrm{dx}\mathrm{dt}
+\int_{\Sigma_T} \gamma( {\theta}_M )| {\theta}_M |^{\ell-2}  {\theta}_M  v \mathrm{ds}\mathrm{dt}  +\nonumber \\
  +
\int_{Q_T} \sigma ( \theta_M ) \Pi (  \theta_M ) \nabla   {\phi}_M  \cdot\nabla v\mathrm{dx}\mathrm{dt}+
\int_{Q_T}   \phi_M \mathbf{j}(  \theta_M ,   \phi_M  ) \cdot\nabla v\mathrm{dx}\mathrm{dt} =\nonumber \\
= 
\int_{\Sigma_T}  h_M v  \mathrm{ds}\mathrm{dt}  ;\qquad \label{tdMj}\\
\int_\Omega 
\mathbf{j}( \theta_M , \phi_M)\cdot \nabla w\mathrm{dx}=\int_{\Gamma_{\rm N}}g w\mathrm{ds}, \qquad \label{phiMj}
\end{eqnarray}
for every $v\in V_{\ell}(Q_T)$ and $w\in V$. 
There exist $\theta,\phi: Q_T\rightarrow\mathbb{R}$ and subsequences of  $(\theta_M,\phi_M)$,
still labelled by  $(\theta_M,\phi_M)$,  weakly convergent in accordance with Proposition \ref{ttphi}.
By Corollary \ref{pphi}, $(\theta,\phi)$ verify the electric equality
\begin{equation}\label{ee}\\
\int_\Omega \mathbf{j}( \theta , \phi)\cdot \nabla w\mathrm{dx}=
\int_{\Gamma_{\rm N}}g w\mathrm{ds}, 
\end{equation}
for every  $w\in V$.

We emphasize that the  weak convergence of
$\mathbf{j}_M$ to $ \mathbf{j}$ in $ \mathbf{L}^2(Q_T)$, taking  (\ref{phim1})-(\ref{phim2}) into account,
is not sufficient to pass to the limit the term $\phi_M \mathbf{j}(  \theta_M ,   \phi_M  ) $.
Moreover, the non smoothness of the coefficients destroy the possibility of obtaining strong convergences of 
$\nabla\theta_M$ and of $\phi_M$.

Thanks to Proposition \ref{pbttm}, we have a.e. pointwise convergence for a subsequence of $\theta_M$,
which we still denote by $\theta_M$.
Considering the assumptions (\ref{kmm})-(\ref{pmax}), the Nemytskii operators are 
continuous due to the Krasnoselski theorem, and applying the Lebesgue dominated convergence theorem,
we obtain
\begin{eqnarray*}
k(  \theta_M )\nabla v\rightarrow k(\theta) \nabla v &\mbox{ in }& \mathbf{L}^2(Q_T); \\
\sigma ( \theta_M )   \alpha_\mathrm{S}(  \theta_M  ) \nabla v\rightarrow
   \sigma ( \theta )  \alpha_\mathrm{S}(  \theta ) \nabla v &\mbox{ in }& \mathbf{L}^2(Q_T); \\
  \sigma ( \theta_M ) \Pi (  \theta_M  ) \nabla v\rightarrow 
   \sigma ( \theta )   \Pi (  \theta ) \nabla v &\mbox{ in }& \mathbf{L}^2(Q_T).
\end{eqnarray*}

Applying (\ref{cotaul}) to the following estimates
\begin{eqnarray*}
 \| T_\mathcal{M}(\phi_M) \nabla \theta_M\|_{2,Q_T}
  \leq \mathcal{M}  \| \nabla \theta_M\|_{2,Q_T} ; \\
 \|  \sigma ( \theta_M ) T_\mathcal{M}(  \phi_M  ) \nabla \phi_M\|_{2,Q_T}\leq
 \sigma^\# \mathcal{M}  \| \nabla \phi_M\|_{2,Q_T} ;\\
 \|  \gamma ( \theta_M ) | \theta_M|^{\ell-2} \theta_M\|_{\ell ',\Sigma_T } \leq \gamma^\# \|  \theta_M\|_{\ell,\Sigma_T } .
\end{eqnarray*}
there exist $\Lambda_1 , \Lambda_2 \in \mathbf{L}^2(Q_T)$
 and $\Lambda_3\in L^{\ell '}(\Sigma_T ) $ such that
\begin{eqnarray*}
  T_\mathcal{M}(  \phi_M  )\nabla \theta_M\rightharpoonup \Lambda_1 &\mbox{ in }& \mathbf{L}^2(Q_T); \\
  \sigma ( \theta_M ) T_\mathcal{M}(  \phi_M  ) \nabla \phi_M\rightharpoonup \Lambda_2 &\mbox{ in }& \mathbf{L}^2(Q_T); \\
 \gamma ( \theta_M ) | \theta_M|^{\ell-2} \theta_M\rightharpoonup \Lambda_3 &\mbox{ in }& L^{\ell '}(\Sigma_T ) .
\end{eqnarray*}
Thus,   we may pass to the limit
in (\ref{tdMj}) as $M$ tends to infinity, concluding that $(\theta,\phi)$ verifies
 \begin{eqnarray}
\int_0^T\int_\Omega  Z v\mathrm{dx}\mathrm{dt} +
\int_{Q_T}  k (\theta)\nabla\theta \cdot \nabla v\mathrm{dx}\mathrm{dt}
 + \nonumber \\ +\int_{Q_T} \sigma(\theta) \Pi (\theta)
\nabla\phi \cdot\nabla v\mathrm{dx} \mathrm{dt}+ 
\int_{Q_T} \sigma ( \theta )  \alpha_\mathrm{S}(\theta ) \Lambda_1 \cdot\nabla v\mathrm{dx}\mathrm{dt} +\nonumber \\
  +
\int_{Q_T} \Lambda_2  \cdot\nabla v\mathrm{dx}\mathrm{dt}+\int_{\Sigma_T}  \Lambda_3 v \mathrm{ds}\mathrm{dt} 
= 
\int_{\Sigma_T}  h v  \mathrm{ds}\mathrm{dt} ,\quad \forall v \in V_\ell (Q_T) . \label{wL123}
\end{eqnarray}

To identify the temperature $\theta$ as a solution, 
we need to identify $\Lambda_1 , \Lambda_2 \in \mathbf{L}^2(Q_T)$  and $\Lambda_3\in L^{\ell '}(\Sigma_T ) $.

To prove that $\Lambda_1=  T_\mathcal{M}(  \phi  )\nabla \theta$, let us consider the Green formula
\[
\int_{Q_T}  T_\mathcal{M}(  \phi_M  )\nabla \theta_M \cdot \mathbf{v}\mathrm{dx}\mathrm{dt}=-
\int_{Q_T[|\phi_M|<\mathcal{M}]} \theta_M  \nabla  \phi_M  \cdot \mathbf{v}\mathrm{dx}\mathrm{dt},
\]
for every $\mathbf{v}\in \mathbf{L}^p(0,T; \mathbf{W}^{1,p}(\Omega))$ such that $\nabla\cdot \mathbf{v}=0$ in $Q_T$
and $\mathbf{v}\cdot \mathbf{n}=0$ in $\partial\Omega\times ]0,T [$. Next we choose the exponent $p>1$
to ensure the meaning of the involved terms. By $\theta_M\in L^\infty (0,T; L^2(\Omega))\cap
L^2(0,T; H^1(\Omega ))$ and $ H^1(\Omega )\hookrightarrow L^{2^*}(\Omega)$ with $2^*$ being the critical Sobolev
exponent, \textit{i.e.} $2^*=2n/(n-2)$ if $n>2$ and any $2^*>1$ if $n=2$, making recourse to the interpolation
with exponents being
\[
\frac{\beta}{2}= \frac{1-\beta}{2}+\frac{\beta}{2^*}=\frac{1}{q}, \qquad (0<\beta <1),
\]
then $\theta_M$ converges to $\theta$ in $L^q(Q_T)$ for every $q<2(n+2)/n$. In particular,
we take $p>n+2$ such that
\[
\frac{1}{2}=\frac{1}{p}+\frac{1}{q} >\frac{1}{p}+\frac{n}{2(n+2)}.
\]
Consequently, we have that  $\theta_M\mathbf{v}$ converges to $\theta\mathbf{v}$ in $\mathbf{L}^2 (Q_T)$.
Therefore,  the uniqueness of the weak limit implies that $\Lambda_1=  T_\mathcal{M}(  \phi  )\nabla \theta$.
In particular, we find 
\[
\int_{Q_T} a(\theta_M,\phi_M)\nabla \theta_M\cdot\nabla v\mathrm{dx}\mathrm{dt}
\build\longrightarrow_{M\rightarrow\infty} \int_{Q_T} a(\theta,\phi) \nabla \theta\cdot\nabla v\mathrm{dx}\mathrm{dt},
 \]
 and consequently (\ref{wL123}) reads
 \begin{eqnarray}
\int_0^T\int_\Omega  Z v\mathrm{dx}\mathrm{dt} +
\int_{Q_T}  a(\theta,\phi)\nabla\theta \cdot \nabla v\mathrm{dx}\mathrm{dt}
  +\int_{Q_T}  \sigma(\theta) \Pi (\theta)
\nabla\phi \cdot\nabla v\mathrm{dx} \mathrm{dt}+ \nonumber \\
  +
\int_{Q_T} \Lambda_2  \cdot\nabla v\mathrm{dx}\mathrm{dt}+\int_{\Sigma_T}  \Lambda_3 v \mathrm{ds}\mathrm{dt} 
= 
\int_{\Sigma_T}  h v  \mathrm{ds}\mathrm{dt} ,\quad \forall v \in V_\ell (Q_T) . \qquad \label{wL23}
\end{eqnarray}

Now, we are in the conditions to identify the limits $\Lambda_2$ and $\Lambda_3$
by making recourse to the Minty argument as follows. 
We rephrase (\ref{jmm}) as
 \begin{eqnarray*}
 \mathcal{J}_M -\mathcal{J}_1^M-\mathcal{J}_2^M -\mathcal{J}_3^M -\mathcal{J}_4^M -\mathcal{J}_5^M + \\
+\int_{\Sigma_T} \left(\gamma(\theta_M) |\theta_{M}|^{\ell-2} \theta_{M} -\gamma(v) |v|^{\ell-2}v\right) 
(\theta_{M} -v)\mathrm{ds}\mathrm{dt} \geq  \\ \geq
(L_1)_\# \int_{Q_T}|\nabla (\theta_M-v) |^2\mathrm{dx} \mathrm{dt}+
(L_2)_\#\int_{Q_T}|\nabla (\phi_M - \phi ) |^2\mathrm{dx}\mathrm{dt}\geq 0,
\end{eqnarray*}
where
 \begin{eqnarray*}
\mathcal{J}_M &:=& \int_{Q_T}
\left( a(\theta_M,\phi_M) |\nabla \theta_M|^2 +\sigma (\theta_M) F (\theta_M,	\phi_M) \nabla \phi_M\cdot \nabla \theta_M
\right)\mathrm{dx}\mathrm{dt} + \\ &&+\int_{Q_T}
\left(\sigma (\theta_M) |\nabla \phi_M|^2 + \sigma (\theta_M)  \alpha_\mathrm{S} (\theta_M) \nabla \theta_M\cdot \nabla \phi_M
\right)\mathrm{dx} \mathrm{dt}+ \\ &&
+\int_{\Sigma_T}\gamma(\theta_M) |\theta_M|^\ell\mathrm{ds}\mathrm{dt}; \\
\mathcal{J}_1^M &:=& 2
\int_{Q_T} a(\theta_M,\phi_M)\nabla \theta_M\cdot\nabla v\mathrm{dx}\mathrm{dt}-
\int_{Q_T} a(\theta_M,\phi_M)|\nabla v|^2\mathrm{dx}\mathrm{dt}; \\
\mathcal{J}_2^M &:=& 2
\int_{Q_T}\sigma(\theta_M)\nabla \phi_M\cdot\nabla \phi \mathrm{dx}\mathrm{dt} -
\int_{Q_T}\sigma(\theta_M) |\nabla \phi |^2\mathrm{dx}\mathrm{dt} ;\\
\mathcal{J}_3^M &:=&
\int_{Q_T} \sigma (\theta_M) \Pi (\theta_M) \Big(\nabla \phi_M \cdot\nabla v +
\nabla \phi \cdot\nabla( \theta_M-v) \Big) \mathrm{dx} \mathrm{dt} ; \\
\mathcal{J}_4^M &:=&
\int_{Q_T} \sigma (\theta_M) T_\mathcal{M} (\phi_M) \Big(\nabla \phi_M \cdot\nabla v +
\nabla \phi  \cdot\nabla( \theta_M-v) \Big) \mathrm{dx} \mathrm{dt} ; \\
\mathcal{J}_5^M &:=&
\int_{Q_T} \sigma (\theta_M)  \alpha_\mathrm{S} (\theta_M)\Big( \nabla \theta_M \cdot\nabla \phi  +
\nabla v\cdot\nabla(\phi_M - \phi ) \Big) \mathrm{dx} \mathrm{dt}.
\end{eqnarray*}

Considering the convergences
 \begin{eqnarray*}
 \mathcal{J}_1^M & \build\longrightarrow_{M\rightarrow\infty} & 2
\int_{Q_T} a(\theta,\phi)\nabla \theta\cdot \nabla v\mathrm{dx}\mathrm{dt} -
\int_{Q_T} a(\theta,\phi) |\nabla v|^2\mathrm{dx} \mathrm{dt}:= \mathcal{J}_1;\\
 \mathcal{J}_2^M & \build\longrightarrow_{M\rightarrow\infty} &
\int_{Q_T}\sigma(\theta) |\nabla \phi |^2\mathrm{dx} \mathrm{dt} := \mathcal{J}_2; \\
 \mathcal{J}_3^M & \build\longrightarrow_{M\rightarrow\infty} &
\int_{Q_T}\sigma(\theta)  \Pi (\theta)  \nabla \phi \cdot\nabla \theta  \mathrm{dx} \mathrm{dt} := \mathcal{J}_3;\\
 \mathcal{J}_4^M & \build\longrightarrow_{M\rightarrow\infty} &
\int_{Q_T} \Lambda_2   \cdot\nabla v\mathrm{dx} \mathrm{dt} +
\int_{Q_T}\sigma(\theta)T_\mathcal{M}(\phi)\nabla \phi  \cdot\nabla( \theta-v)\mathrm{dx} \mathrm{dt} := \mathcal{J}_4;\\
 \mathcal{J}_5^M & \build\longrightarrow_{M\rightarrow\infty} &
\int_{Q_T}\sigma(\theta)  \alpha_\mathrm{S} (\theta)\nabla \theta \cdot\nabla \phi \mathrm{dx}\mathrm{dt} 
:= \mathcal{J}_5,
\end{eqnarray*}
  we deduce
 \begin{eqnarray*}
\lim_{ M\rightarrow \infty}\mathcal{J}_M\geq
\int_{\Sigma_T}\Lambda_3 v\mathrm{ds} \mathrm{dt}+
\int_{\Sigma_T}\gamma(v)|v|^{\ell-2} v(\theta -v)\mathrm{ds}\mathrm{dt}+ \\ + \mathcal{J}_1+
\mathcal{J}_2+ \mathcal{J}_3 +\mathcal{J}_4 +\mathcal{J}_5.
\end{eqnarray*}
We continue the Minty argument by taking in (\ref{tdMj}) and (\ref{phiMj}), respectively, the test function $v=\theta_M$ and
  the test function $w=\phi_M$, in (\ref{wL23}) the test function $v=\theta$,
 and in (\ref{ee}) the test function $w=\phi$, we deduce
 \begin{eqnarray*}
\lim_{M\rightarrow \infty}\mathcal{J}_M = -
\int_{Q_T} Z\theta \mathrm{dx} \mathrm{dt}+\int_{\Sigma_T}  h\theta \mathrm{ds} \mathrm{dt}
+\int_0^T \int_{\Gamma_\mathrm{N}}  g \theta \mathrm{ds}\mathrm{dt}
=\\
=\int_{Q_T} a( \theta,\phi ) |\nabla \theta |^2\mathrm{dx}\mathrm{dt}+
\int_{Q_T} \sigma (\theta) \Pi (\theta)
\nabla\phi \cdot\nabla \theta\mathrm{dx} \mathrm{dt} + \\ +\int_{Q_T} \Lambda_2\cdot\nabla \theta\mathrm{dx}\mathrm{dt}
+\int_{\Sigma_T} \Lambda_3 \theta \mathrm{ds}\mathrm{dt}+
 \int_{Q_T} \mathbf{j}( \theta , \phi) \cdot \nabla \phi\mathrm{dx} \mathrm{dt}.
\end{eqnarray*}

Gathering the above two relations, we find
 \begin{eqnarray}
\int_{Q_T} a( \theta,\phi )|\nabla (\theta- v)|^2\mathrm{dx} \mathrm{dt}+
\int_{Q_T} \Big(\Lambda_2 -\sigma(\theta)T_\mathcal{M}(\phi)\nabla \phi \Big)
 \cdot\nabla( \theta-v)\mathrm{dx} \mathrm{dt} + \nonumber\\+
\int_{\Sigma_T} \left(\Lambda_3 -\gamma(v)|v|^{\ell-2}v \right) (\theta-v) \mathrm{ds}\mathrm{dt} \geq 0.\label{minty}
\end{eqnarray} 
Next, taking $v=\theta-\delta \varphi$, with $\varphi\in\mathcal{D}(Q_T)$,  and after
  dividing by $\delta>0$, we arrive to
 \begin{eqnarray*}
\delta\int_{Q_T} a( \theta,\phi )|\nabla \varphi|^2\mathrm{dx} \mathrm{dt}+
\int_{Q_T} \Big(\Lambda_2 -\sigma(\theta)T_\mathcal{M}(\phi)\nabla \phi\Big)
 \cdot\nabla\varphi\mathrm{dx} \mathrm{dt}  \geq 0.
\end{eqnarray*} 
Finally letting $\delta\rightarrow 0^+$ then we obtain $\Lambda_2 =\sigma(\theta)T_\mathcal{M}(\phi)\nabla \phi$.

We conclude the Minty argument by taking $v=\theta-\delta \varphi$, with $\varphi\in\mathcal{D}(\Sigma_T)$,  
in (\ref{minty}).
 After dividing by $\delta>0$,
and finally letting $\delta\rightarrow 0^+$ we arrive to
 \begin{eqnarray*}
\int_{\Sigma_T}(\Lambda_3 -\gamma(\theta) |\theta|^{\ell-2}\theta )\varphi\mathrm{ds}\mathrm{dt} \geq 0, \quad\forall
\varphi \in\mathcal{D}(\Sigma_T),
\end{eqnarray*}
which implies that $\Lambda_3 =\gamma(\theta)|\theta|^{\ell-2}\theta$.

Therefore, the weak formulation (\ref{pbtt}) yields 
concluding the proof of Theorem \ref{tmain}.

\end{document}